\documentclass[a4paper,10pt,reqno]{article}
\usepackage{float}
\usepackage[export]{adjustbox}
\usepackage{silence}
\usepackage{sectsty}
\usepackage{fullpage}
\usepackage[normalem]{ulem}
\usepackage[table]{xcolor} 
\usepackage{multirow}
\usepackage{mathtools}
\usepackage{amsmath}
\usepackage{amsthm}
\usepackage{adjustbox}
\usepackage{amsfonts}
\usepackage{amssymb}
\usepackage{multirow}
\usepackage{array}
\usepackage{makeidx}
\usepackage{xspace}
\usepackage{esvect}
\usepackage[most]{tcolorbox}
\usepackage[unicode]{hyperref}
\usepackage{cite}
\usepackage{xspace} 
\usepackage{float}
\usepackage{tikz}
\usepackage{placeins}
\usepackage{babel}
\usepackage[caption=false]{subfig}
\usepackage{graphicx}
\usepackage{colortbl} 
\definecolor{lightgray}{rgb}{0.9,0.9,0.9} %
\usepackage[left=1.4cm,right=1.4cm,top=1.5cm,bottom=1.5cm]{geometry}
\usepackage{amsthm,amssymb,mathrsfs, pstricks,booktabs}
\allowdisplaybreaks
\newtheorem{thm}{Theorem}[section]

\usepackage{latexsym}
\usepackage{graphicx}
\usepackage{epstopdf}
\usepackage{play}
\usepackage{epsfig}
\usepackage[nottoc]{tocbibind}
\newtheorem{lemma}{Lemma}[section]

\usepackage[utf8]{inputenc}
\usepackage[version=4]{mhchem}
\usepackage{stmaryrd}
\usepackage{bbold}

\usepackage{sectsty}
\sectionfont{\centering}
\makeatletter
\renewenvironment{proof}[1][\proofname]{%
  \par\pushQED{\qed}\normalfont%
  \topsep6\p@\@plus6\p@\relax
\trivlist\item[\hskip\labelsep\bfseries#1\@addpunct{.}]%
  \ignorespaces
}{%
  \popQED\endtrivlist\@endpefalse
}
\makeatother
\usepackage{blindtext}

\usepackage{hyperref}
\hypersetup{
    linkcolor=black,
    filecolor=black,      
    urlcolor=black,
    pdfpagemode=FullScreen,}
\begin{document}
\begin{center}
    \textbf{\large Soliton Dynamics and Modulation Instability in the (3+1)-dimensional Zakharov–Kuznetsov equation: A Lie Symmetry Approach}
\end{center}
\begin{center}
\author[{\small \textbf{Anshika Singhal}\textsuperscript{1}, \textbf{Urvashi Joshi}\textsuperscript{2}, \textbf{Rajan Arora}\textsuperscript{3}\\ \textsuperscript{1,2,3} \text{Department of Applied Mathematics and Scientific Computing,}\\ \text{Indian Institute of Technology, Roorkee, India}}\\
\end{center}
\begin{center}
\textsuperscript{1}anshika\textunderscore s@amsc.iitr.ac.in,\textsuperscript{2}urvashi\textunderscore j@amsc.iitr.ac.in.\textbf
\\\textsuperscript{3}rajan.arora@as.iitr.ac.in.\textbf{(Corresponding Author)}
\end{center}
\section*{Abstract}
\noindent The core focus of this research work is to obtain invariant solutions and conservation laws of the (3+1)-dimensional Zakharov–Kuznetsov (ZK) equation, a higher-dimensional generalization of the  Korteweg–de Vries (KdV) equation, which describes the phenomenon of wave stability and soliton propagation. Lie symmetry analysis has been applied to derive infinitesimal generators and classify the optimal subalgebras. Utilizing them, we construct exact invariant solutions that reveal how waves retain their shape as they travel, how they interact in space, and the impact of magnetic fields on wave propagation. Further, by implementing the traveling wave transformation, we derive additional exact solutions, including those exhibiting kink-type solitons. It also concludes with the conservation laws and the non-linear self-adjointness property. Our examination is broadened to cover modulation instability and gain spectrum. To contextualize our results, we compare the solutions obtained from the Lie symmetry method with those derived using the modified simple equation (MSE) method and other symbolic techniques reported in recent literature. To validate the accuracy of the analytical solutions obtained via Lie symmetry, a numerical method is implemented. A review of the ZK equation's physical background and mathematical complexity is explored, emphasizing the limitations of symbolic approaches. \\

\noindent 
\textbf{Keywords:}\hspace{0.1cm}Lie symmetry analysis, Optimal subalgebra, Parameter analysis, Traveling wave solution, Modulation instability, Self-adjointness, Conservation laws, Numerical method, MSE method.
\section{Introduction}
\quad \quad The Zakharov-Kuznetsov equation, a higher-dimensional generalization of the Korteweg-de Vries (KdV) equation, plays an important role in space plasmas and astrophysics, thereby helping in understanding the anisotropic nature of the propagation of waves along the magnetic field lines. It was first derived by Zakharov and Kuznetsov \cite{zakharov1974three} to describe the evolution of weakly nonlinear ion-acoustic waves in a plasma composed of cold ions and hot isothermal electrons under the influence of a uniform magnetic field. Due to its broad relevance in fluid dynamics, plasma physics, and nonlinear waves, the ZK equation \cite{wazwaz2010partial} has been the focus of research, particularly in analytical and numerical perspectives. The general form of the three-dimensional ZK equation \cite{islam2015generalized} is expressed as 
\begin{equation}\label{11}
u_{t}+auu_{x}+u_{x x}+u_{y y}+u_{z z}=0,
\end{equation}
where $a$ is the positive constant representing the strength of nonlinearity in plasma. The function $u=u(x, y, z, t)$ represents the amplitude of the wave, indicating how the height of the wave changes in space with the coordinates $(x, y, z)$ and the time $t$. $u_t$ represents the time evolution of the wave. $u_{x x},u_{y y}$ and $u_{z z} $ are the second-order spatial derivatives that represent how the wave spreads out in space.

In contrast to the  Kadomtsev–Petviashvili (KP) equation, the ZK equation lacks integrability through the inverse scattering method, thereby restricting the analytical techniques for extracting exact solutions. Therefore, many direct and semi-analytical techniques such as the modified simple equation (MSE) method \cite{khan2014exact}, the Riccati method \cite{niwas2024exploring}, the ($\frac{G^\prime}{G}$)-expansion method \cite{mathanaranjan2022effective}, the Bäcklund transformations\cite{zhou2022auto}, and variational iteration methods have been developed to obtain solitary and periodic solutions. These symbolic methods commonly employ assumed ansatz forms or traveling wave transformations, which are effective in generating particular solutions but generally fail to reveal the underlying symmetries or to classify the families of invariant solutions.

 Sophus Lie \cite{fritzsche1999sophus,helgason1992sophus,schwarz1988symmetries} introduced the transformation group method in the late 19th century to find invariant solutions, providing a systematic approach to identifying the continuous symmetries of nonlinear differential equations. By determining infinitesimal generators, optimal subalgebras, and symmetry reductions, one can derive the exact invariant solutions and the conserved vectors. Several top-notch research articles \cite{lou2005non,zhai2019lie,gazizov1998lie} and textbooks \cite{kosmann2010groups,cantwell2004introduction,bluman1990simplifying} address this method and its significance in various physical problems. Despite being a classical technique, Lie symmetry methods are still underexploited in the context of higher-dimensional nonlinear evolution equations like the (3+1)-dimensional ZK equation.

 Addressing this gap, this study applies a comprehensive analysis of the (3+1)-dimensional ZK equation. We find Lie symmetries and infinitesimal generators for which the equation remains invariant. Using infinitesimal generators, we derive seven vector fields, which reduces the system of PDEs into ODEs to derive invariant solutions concerning the invariants. Traveling wave solutions have been constructed to derive solutions as kink-type solitons. Modulation instability has been evaluated with the help of the dispersion relation. Conservation laws have been constructed using the non-linear self-adjointness property \cite{tracina2014nonlinear}. Conservation laws can be derived by many methods like the Direct method, Noether's theorem \cite{kosmann2011noether,byers1998noether}, the multiplier method \cite{naz2012conservation,yasar2017symmetries}, the variational method, and many more, but Ibragimov's \cite{ibragimov2011nonlinear,ibragimov2013nonlinear,anco2017incompleteness} method is considered a more effective method to construct conservation laws as it applies to both variational and non-variational systems.
\begin{center}
The main pioneering outcomes of this paper are mentioned below:
\end{center}
\begin{enumerate} 
\item 
Using Lie symmetry analysis, we derive Lie symmetries, a commutator table, and an optimal system for the ZK equation.\item Traveling wave solutions and modulation instability have been analyzed to further understand the equation's behavior. \item 3D visualizations have been created that highlight the behavior of wave amplitude.\item To acknowledge physical properties, conservation laws are derived.
\item A numerical method is implemented to validate the analytical solutions derived from Lie symmetry. \item We compare our results with recent symbolic solutions from recent literature, emphasizing the structural understanding offered
 by the Lie symmetry framework.
\end{enumerate}
\subsection{Outline}
\quad \quad The paper structure is highlighted as follows: In Section \eqref{2}, Lie symmetries, infinitesimal generators, commutator tables, and vector fields are constructed. Section \eqref{3} outlines the symmetry group of the model. Section \eqref{4} gives the optimal subalgebra of the generated vector fields. Section \eqref{5} is devoted to finding invariant solutions of the system by symmetry reductions. Section \eqref{6} outlines the behavior of wave amplitude using 3D visualizations. Section \eqref{7} deals with the traveling wave solution. In Section \eqref{8}, we established the modulation instability of the governing equation. Sections \eqref{9} and \eqref{10} deal with the numerical method and its comparison with the analytical method. Within Section \eqref{11a}, we establish the property of non-linear self-adjointness. Section \eqref{12} is dedicated to deriving conservation laws using the property of non-linear self-adjointness. Section \eqref{13} is devoted to the comparison with the MSE method. Conclusions are provided in Section \eqref{14}.
\section{A Preliminary Analysis of Lie Symmetry}\label{2}
A one-parameter Lie group of transformation for Eq. \eqref{11}, is expressed as:  
\begin{equation}
\begin{split}
\hat{s}_{i}=s_{i}+\varepsilon  \xi_{i}(x, y, z, t, u)+O\left(\varepsilon^{2}\right)\text{,}
\end{split}
\end{equation}
where $s_{i}$ is $x, y, z, t$ and $u$, respectively, for $i=x, y, z, t, u$ and $\varepsilon$ is the group parameter and $\xi_{x}, \xi_{y}, \xi_{z}, \xi_{t}, \xi_{u}$  are the infinitesimal generators.
Thus, the vector field derived from the Lie group of transformation is as follows:
\begin{equation}
\mathbb{Z}=\xi_{i}(x, y, z, t, u) \partial_{i}\text{,}
\end{equation}
for $i= x, y, z, t, u$\text{.}
For the nonlinear ZK equation, $\mathbb{Z}$ must satisfy the invariance condition:
\begin{equation}
{Pr}^{(2)} \mathbb{Z}(\Delta)|_{\Delta=0} =0,
\end{equation}
where ${Pr}^{(2)}$ is the second-order prolongation of $\mathbb{Z}$ for Eq. \eqref{11}, as follows:\\
\begin{equation}
{Pr}^{(2)}=  \mathbb{Z}+\xi_{t}^{u} \frac{\partial}{\partial u_{t}}+\xi_{x}^{u} \frac{\partial}{\partial u_{x}}+\xi_{x x}^{u} \frac{\partial}{\partial u_{x x}}+\xi_{y}^{u} \frac{\partial}{\partial u_{y}}+\xi_{y y}^{u} \frac{\partial}{\partial u_{y y}}+\xi_{z}^{u} \frac{\partial}{\partial u_{z}}+\xi_{z z}^{u} \frac{\partial}{\partial u_{z z}} +\xi_{u} \frac{\partial}{\partial u}\text{.}
\end{equation}
Owing to the extensive calculations, we relied on Maple to derive the following system of PDEs:
\begin{equation}\label{k}
\begin{split}
&\xi_{tt}^{t}=\xi_{tt}^{x}=\xi_{yy}^{z}=\xi_{u}^{t}=\xi_{x}^{t}=\xi_{y}^{t}=\xi_{z}^{t}=\xi_{u}^{x}=0, \quad \xi_{x}^{x} = \frac{\xi_{t}^{t}}{2}, \quad \xi_{y}^{x} = \xi_{z}^{x} = \xi_{t}^{y} = \xi_{u}^{y} = \xi_{x}^{y} = 0, \\
&\xi_{y}^{y} = \frac{\xi_{t}^{t}}{2}, \quad \xi_{z}^{y} = -\xi_{y}^{z}, \xi_{t}^{z} = \xi_{u}^{z} = \xi_{x}^{z} = 0, \quad \xi_{z}^{z} = \frac{\xi_{t}^{t}}{2}, \quad \xi_{u} = \frac{-\left(\xi_{t}^{t} au + 2\xi_{t}^{x}\right)}{2a}.
\end{split}
\end{equation}
On solving  Eq. \eqref{k}, we derive the infinitesimals as:
\begin{equation}
\begin{split}
&\xi_{x}=\frac{c_{1}x}{2}+c_{5} t+c_{6}\text{,} \quad\xi_{y}=\frac{c_{1}y}{2}+{c_{3}z}+c_{4}\text{,} \quad\xi_{z}=-{c_{3}y}+\frac{c_{1}z}{2}+c_{7}\text{,}\quad\xi_{t}={c_{1}t}+c_{2} \text{,}\quad\xi_{u}=\frac{-c_{1}u}{2}+\frac{c_{5}}{a}\text{,} \\ 
\end{split}
\end{equation}
where $c_{i}$'s , $i=1,2,3,4,5,6,7$, are the arbitrary constants.\\

\begin{table}[h]
    \centering
    \renewcommand{\arraystretch}{1.5}
    \begin{tabular}{c c l}
        \hline
        Conditions & Generators & Symmetries \\ 
        \hline
        $c_1 \neq 0$ & $\mathfrak{D}_1 = \frac{x}{2} \partial_{x}+ \frac{y}{2} \partial_{y}+ \frac{z}{2} \partial_{z}\ + {t}\partial_{t}-\frac{u}{2}\partial_{u}$ & Dilation in $x,y,z,t, u$ \\ 
        $c_2 \neq 0$ & $\mathfrak{D}_2 =  \frac{\partial}{\partial t} $ & Translation in $t$ \\ 
        $c_3 \neq 0$ & $\mathfrak{D}_3 = z\frac{\partial}{\partial y}-y\frac{\partial}{\partial z} $ & Rotation \\ 
        $c_4 \neq 0$ & $\mathfrak{D}_4 =  \frac{\partial}{\partial y} $ & Translation in $y$ \\ 
        $c_5 \neq 0$ & $\mathfrak{D}_5 =t \frac{\partial}{\partial x}+\frac{1}{a}\frac{\partial}{\partial u}$ & Galilean boost \\ 
        $c_6 \neq 0$ & $\mathfrak{D}_6 =  \frac{\partial}{\partial x} $ &Translation in $x$ \\ 
        $c_7 \neq 0$ & $\mathfrak{D}_7 = \frac{\partial}{\partial z}$ & Translation in $z$ \\ 
        \hline
    \end{tabular}
  
    \caption{Infinitesimal Generators and Their Symmetries.}
\end{table}
Now, we obtain Table \ref{Table 1} which shows the commutation relations between the Lie algebras ${\mathfrak{D}_{1},\mathfrak{D}_{2},\mathfrak{D}_{3},\mathfrak{D}_{4},\mathfrak{D}_{5},\mathfrak{D}_{6}}$,\\
${\mathfrak{D}_{7}}$, which has entries like [$\mathfrak{D}_i,\mathfrak{D}_j$]=$\mathfrak{D}_{i}\mathfrak{D}_{j}-\mathfrak{D}_{j}\mathfrak{D}_{i}$.
\renewcommand{\arraystretch}{1.5}
\begin{table}[H]
    \centering
    \begin{tabular}{l|lllllll}
    \hline
        $\star$ & $\mathfrak{D}_1$ & $\mathfrak{D}_2$ & $\mathfrak{D}_3$ & $\mathfrak{D}_4$ & $\mathfrak{D}_5$ & $\mathfrak{D}_6$ & $\mathfrak{D}_7$  \\  \hline
        $\mathfrak{D}_1$ & 0 & $-\mathfrak{D}_2$ & 0 & ${\frac{-\mathfrak{D}_4}{2}}$ & ${\frac{\mathfrak{D}_5}{2}}$ & ${\frac{-\mathfrak{D}_6}{2}}$ & ${\frac{-\mathfrak{D}_7}{2}}$ \\ 
       $\mathfrak{D}_2$ & $\mathfrak{D}_2$ & 0 & 0 & 0 & $\mathfrak{D}_6$ & 0 & 0\\
        $\mathfrak{D}_3$ & 0 & 0 & 0 & $\mathfrak{D}_7$ & 0 & 0 & -$\mathfrak{D}_4$  \\ 
        $\mathfrak{D}_4$ & ${\frac{\mathfrak{D}_4}{2}}$ & 0 &$-\mathfrak{D}_7$ & 0 & 0 & 0 & 0 \\
        $\mathfrak{D}_5$ & ${\frac{-\mathfrak{D}_5}{2}}$ & -$\mathfrak{D}_6$ & $0$ & 0 & 0 & 0 & 0 \\ 
        $\mathfrak{D}_6$ &  ${\frac{\mathfrak{D}_6}{2}}$ & 0 & $0$ & 0 & 0 & 0 & 0 \\ 
        $\mathfrak{D}_7$ &  ${\frac{\mathfrak{D}_7}{2}}$ & 0 & $\mathfrak{D}_4$ & 0 & 0 & 0 & 0 \\ \hline
    \end{tabular}
    \caption{Commutator Table.}
    \label{Table 1}
\end{table}
\section{Lie Symmetry Group for ZK Equation} \label{3}
\quad \quad Let us consider the one-parameter group $\tilde{\mathcal{V}_i^\varepsilon}$ generated by $\mathfrak{D}_{i}\text{,} \hspace{0.2cm} i=1,2, \ldots, 7$\text{,}  as
\begin{equation}
\tilde{\mathcal{V}_i^\varepsilon}:(x, y, z, t, u) \rightarrow(\ddot{x}, \ddot{y},\ddot{z}, \ddot{t}, \ddot{u}) \text {, }
\end{equation}
that generates the invariant solutions. We must figure out the following IVP:
\begin{equation}
\begin{aligned}
& \frac{d}{d \varepsilon}(\ddot{x}, \ddot{y},\ddot{z}, \ddot{t}, \ddot{u})=(\xi_{x}, \xi_{y}, \xi_{z}, \xi_{t}, \xi_{u})\hspace{0.1cm}\text { with } \hspace{0.1cm}\left.(\ddot{x}, \ddot{y},\ddot{z}, \ddot{t}, \ddot{u})\right|_{\varepsilon=0}=(x, y, z, t, u) \text{.}
\end{aligned}
\end{equation}
\begin{thm}
If the function $u= \mathcal{U}(x, y, z, t)$ satisfies Eq. \eqref{11}, then the group action  yields the following one-parameter family of solutions:
$$
\begin{array}{ll}
& \tilde{\mathcal{V}_1^\varepsilon}: u= e^{\frac{-\varepsilon}{2}}\mathcal{U}\left(x e^{\frac{-\varepsilon}{2}}, y e^{\frac{-\varepsilon}{2}},z e^{-\frac{\varepsilon}{2}} , t e^{-\varepsilon}\right) \text{,}\quad
\tilde{\mathcal{V}_2^\varepsilon}: u=\mathcal{U}(x, y, z, t-\varepsilon) \text{,} \quad
 \tilde{\mathcal{V}_3^\varepsilon}: u=\mathcal{U}(x, y-\varepsilon z,z+\varepsilon y , t)\text{,}\\
& \tilde{\mathcal{V}_4^\varepsilon}: u=\mathcal{U}(x, y-\varepsilon, z, t) \text{,} \quad
 \tilde{\mathcal{V}_5^\varepsilon}: u=\mathcal{U}(x-\varepsilon t, y, z, t)-\frac{\varepsilon}{a}  \text{,} \quad
 \tilde{\mathcal{V}_6^\varepsilon}: u=\mathcal{U}(x-\varepsilon, y, z, t) \text{,} \quad
\tilde{\mathcal{V}_7^\varepsilon}: u=\mathcal{U}(x, y, z-\varepsilon, t)  \text{.}
\end{array}
$$
\end{thm}
\section{One-dimensional Optimal System for ZK Equation} \label{4}
\quad \quad Prior to building the lie algebra $\mathbb{L}^{7}$, we will determine the invariants and adjoint matrix. 
\subsection{Calculation of Invariants}
\quad \quad Let us consider two representative elements of $\mathbb{L}^{7}$ as:
$\mathcal{H}=\sum_{i=1}^{7} \mathfrak{l}_{i} \mathfrak{D}_{i} \text { and}$  $\mathcal{J}=\sum_{j=1}^{7} \beta_{j} \mathfrak{D}_{j} $\text {,}
the adjoint action $\operatorname{Ad}(\exp (\varepsilon \mathcal{J}))$ on $\mathcal{H}$ is given by:
$$
\begin{aligned}
\operatorname{Ad}(\exp (\varepsilon\mathcal{J}) \mathcal{H}) & =\mathcal{H}-\varepsilon[\mathcal{J}, \mathcal{H}]+\frac{1}{2 !} \varepsilon^{2}[\mathcal{J},[\mathcal{J}, \mathcal{H}]] \ldots=e^{-\varepsilon \mathcal{J}} \mathcal{H} e^{\varepsilon \mathcal{J}} \text{.}\\
&=  \left(\mathfrak{l}_{1}\mathfrak{D}_{1}+\mathfrak{l}_{2} \mathfrak{D}_{2}+\ldots+\mathfrak{l}_{7} \mathfrak{D}_{7}\right)-\varepsilon\left[\lambda_{1} \mathfrak{D}_{1}+\lambda_{2} \mathfrak{D}_{2}+\ldots+\lambda_{7} \mathfrak{D}_{7}\right]+O\left(\varepsilon^{2}\right) \text{,}
\end{aligned}
$$
where $\lambda_{i}= \lambda_{i}\left(\mathfrak{l}_{1}, \mathfrak{l}_{2}, \ldots, \mathfrak{l}_{7}, \beta_{1}, \beta_{2}, \ldots, \beta_{7}\right), \hspace{0.1cm} i=1,2,3,4,5,6,7$, which can be solved by commutator table, and for invariance,
\begin{equation}\label{q}
\mathcal{O}\left(\mathfrak{l}_{1}, \mathfrak{l}_{2}, \ldots, \mathfrak{l}_{7}\right)=\mathcal{O}\left(\mathfrak{l}_{1}-\varepsilon \lambda_{1}, \mathfrak{l}_{2}-\varepsilon \lambda_{2}, \ldots, \mathfrak{l}_{7}-\varepsilon \lambda_{7}\right)\text{.} 
\end{equation}
On applying Taylor's expansion to the right-hand side of Eq., we obtain
\begin{equation}\label{l}
\lambda_{1}  \frac{\partial\mathcal{O}}{\partial \mathfrak{l}_{1}}+\lambda_{2}  \frac{\partial \mathcal{O}}{\partial \mathfrak{l}_{2}}+\ldots+\lambda_{7}  \frac{\partial \mathcal{O}}{\partial \mathfrak{l}_{7}}=0\text{,} 
\end{equation}
$\text{where}\hspace{0.2cm}\lambda_{1}=0\text{,}\quad \lambda_{2}=-\mathfrak{l}_{2}\beta_{1}+\mathfrak{l}_{1} \beta_{2}\text{,}\quad \lambda_{3}=0\text{,}\quad \lambda_{4}=\frac{-\mathfrak{l}_{4} \beta_{1}}{2} +\frac{\mathfrak{l}_{1} \beta_{4}}{2} -\mathfrak{l}_{7}\beta_{3} +\mathfrak{l}_{3} \beta_{7}\text{,}\quad \lambda_{5}=\frac{\mathfrak{l}_{5}\beta_{1}}{2} -\frac{\mathfrak{l}_{1} \beta_{5}}{2}\text{,}\quad \lambda_{6}=\frac{-\beta_{1}\mathfrak{l}_{6}}{2} +\frac{\beta_{6}\mathfrak{l}_{1}}{2} -\beta_{5}\mathfrak{l}_{2} +\beta_{2}\mathfrak{l}_{5}\text{,}\quad \lambda_{7}=\frac{-\beta_{1}\mathfrak{l}_{7}}{2} +\frac{\beta_{7}\mathfrak{l}_{1}}{2} +\beta_{3}\mathfrak{l}_{4} -\beta_{4}\mathfrak{l}_{3} \text{.}$ Substituting these values of $\lambda_{i}\hspace{0.1cm}\text{,}\hspace{0.1cm}  1 \leq i \leq 7$ in Eq. \eqref{l} and extracting the coefficients $\forall$ $\beta_{j}\hspace{0.1cm}\text{,}\hspace{0.1cm}  1 \leq j \leq 7$, the system of first-order PDEs can be obtained as follows:
\begin{equation}
\begin{split}
&\beta_{1}:-\mathfrak{l}_{2}\frac{\partial \mathcal{O}}{\partial \mathfrak{l}_{2}}-\frac{\mathfrak{l}_{4}}{2} \frac{\partial \mathcal{O}}{\partial \mathfrak{l}_{4}}+\frac{\mathfrak{l}_{5}}{2} \frac{\partial \mathcal{O}}{\partial \mathfrak{l}_{5}}-\frac{\mathfrak{l}_{6}}{2} \frac{\partial \mathcal{O}}{\partial \mathfrak{l}_{6}}-\frac{\mathfrak{l}_{7}}{2} \frac{\partial \mathcal{O}}{\partial \mathfrak{l}_{7}}=0 \text{,} \quad
\beta_{2} : \mathfrak{l}_{1} \frac{\partial \mathcal{O}}{\partial \mathfrak{l}_{2}}+ \mathfrak{l}_{5} \frac{\partial \mathcal{O}}{\partial \mathfrak{l}_{5}}=0\text{,}\\
&\beta_{3}:-\mathfrak{l}_{7}\frac{\partial \mathcal{O}}{\partial \mathfrak{l}_{4}}+\mathfrak{l}_{4} \frac{\partial \mathcal{O}}{\partial \mathfrak{l}_{7}}=0 \text{,} \quad
\beta_{4} :\frac{\mathfrak{l}_{1}}{2} \frac{\partial \mathcal{O}}{\partial \mathfrak{l}_{4}}-\mathfrak{l}_{3}  \frac{\partial \mathcal{O}}{\partial \mathfrak{l}_{7}}=0\text{,} \quad
\beta_{5} : -\frac{\mathfrak{l}_{1}}{2} \frac{\partial \mathcal{O}}{\partial \mathfrak{l}_{5}}-\mathfrak{l}_{2} \frac{\partial \mathcal{O}}{\partial \mathfrak{l}_{6}}=0\text{,} \\
& \beta_{6} : \frac{\mathfrak{l}_{1}}{2} \frac{\partial \mathcal{O}}{\partial \mathfrak{l}_{6}}=0\text{,} \quad
\beta_{7} :\mathfrak{l}_{3} \frac{\partial \mathcal{O}}{\partial \mathfrak{l}_{4}} +\frac{\mathfrak{l}_{1}}{2} \frac{\partial \mathcal{O}}{\partial \mathfrak{l}_{7}}=0\text{.}
\end{split}
\end{equation}
For the solution of the above system of PDEs, one can find that
\begin{equation}
\mathcal{O}\left(\mathfrak{l}_{1}, \mathfrak{l}_{2}, \mathfrak{l}_{3}, \mathfrak{l}_{4}, \mathfrak{l}_{5}, \mathfrak{l}_{6}, \mathfrak{l}_{7}\right)=\mathfrak{M}\left(\mathfrak{l}_{1}, \mathfrak{l}_{3}\right)\text{,}
\end{equation}
where $\mathfrak{M}$ can be chosen as an arbitrary function.
\subsection{Killing Form}
\begin{thm}
The Killing form related to Lie algebra $\mathbb{L}^{7}$ is
\begin{equation}
    \mathcal{K}(\mathfrak{D}, \mathfrak{D})=2\left(\mathfrak{l}_{1}^{2}-\mathfrak{l}_{3}^{2}\right)\text{.}   
 \end{equation}
\end{thm}
\begin{proof}
The Killing form associated with the Lie-algebra $L^{7}$ is defined explicitly as follows, following the definition used in \cite{sharma2023invariance}:
$$\mathcal{K}(\mathfrak{D},\mathfrak{D})=Trace(\operatorname{ad(\mathfrak{D})} \circ \operatorname{ad(\mathfrak{D})})\text{,} $$
\end{proof}
where
$$\begin{aligned}
& \operatorname{Ad}(\mathfrak{D})=\left[\begin{array}{ccccccc}
0 & 0 & 0 & 0 & 0 & 0 & 0\\
-\mathfrak{l}_{2} & \mathfrak{l}_{1} & 0 & 0 & 0 & 0 & 0\\
0 & 0 & 0 & 0 & 0 & 0 & 0\\
-\frac{\mathfrak{l}_{4}}{2} & 0 & -\mathfrak{l}_{7} &  \frac{\mathfrak{l}_{1}}{2} & 0 & 0 & \mathfrak{l}_{3} \\
\frac{\mathfrak{l}_{5}}{2} & 0 & 0 & 0 & -\frac{\mathfrak{l}_{1}}{2} & 0 & 0 \\
-\frac{\mathfrak{l}_{6}}{2} & \mathfrak{l}_{5} & 0 & 0 & -\mathfrak{l}_{2} & \frac{\mathfrak{l}_{1}}{2} & 0 \\
-\frac{\mathfrak{l}_{7}}{2} & 0 & \mathfrak{l}_{4} &  -\mathfrak{l}_{3} & 0 & 0 &\frac{\mathfrak{l}_{1}}{2}  \\
\end{array}\right]\text{.}
\end{aligned} $$ 
\subsection{Construction of Adjoint Transformation Matrix}
\quad \quad It is now essential to construct the general adjoint transformation matrix $\mathcal{A}$, which comprises individual matrices of the adjoint actions $\mathfrak{R}_{1}, \mathfrak{R}_{2}, \ldots \mathfrak{R}_{7}$ concerning $\mathfrak{D}_{1}, \mathfrak{D}_{2}, \ldots \mathfrak{D}_{7}$ to $\mathcal{A}$. Let $\varepsilon_{i}\text{,}  \hspace{0.1cm}i=1,2,3,4,5,6,7$, be real constants and $g=e^{\varepsilon_{i} \mathfrak{D}_{i}}$, then we get
\begin{alignat*}{2}
\mathfrak{R}_{1}= &
\begin{pmatrix}
1 & 0 & 0 & 0 & 0 & 0 & 0 \\
0 & e^{\varepsilon_{1}} & 0 & 0 & 0 & 0 & 0\\
0 & 0 & 1 & 0 & 0 & 0 & 0\\
0 & 0 & 0 &  {e^\frac{\varepsilon_{1}}{2}} & 0 & 0 & 0\\
0 & 0 & 0 & 0 & {e^\frac{-\varepsilon_{1}}{2}} & 0 & 0\\
0 & 0 & 0 & 0 & 0 & {e^\frac{\varepsilon_{1}}{2}} & 0 \\
0 & 0 & 0 & 0 & 0 & 0 & {e^\frac{\varepsilon_{1}}{2}} \\
\end{pmatrix}\text{,}  \qquad&
\mathfrak{R}_{2}= & 
\begin{pmatrix}
1 & -\varepsilon_{2} & 0 & 0 & 0 & 0 & 0\\
0 & 1 & 0 & 0 & 0 & 0 & 0\\
0 & 0 & 1 & 0 & 0 & 0 & 0\\
0 & 0 & 0 & 1 & 0 & 0 & 0\\
0 & 0 & 0 & 0 & 1 & -\varepsilon_{2} & 0\\
0 & 0 & 0 & 0 & 0 & 1 & 0\\
0 & 0 & 0 & 0 & 0 & 0 & 1\\
\end{pmatrix}\text{,} 
\end{alignat*}
\begin{alignat*}{3}
\mathfrak{R}_{3}= &
\begin{pmatrix}
1 & 0 & 0 & 0 & 0 & 0 & 0\\
0 & 1 & 0 & 0 & 0 & 0 & 0\\
0 & 0 & 1 & 0 & 0 & 0 & 0\\
0 & 0 & 0 & \cos{\varepsilon_{3}} & 0 & 0 & -\sin{\varepsilon_{3}} \\
0 & 0 & 0 & 0 & 1 & 0 & 0\\
0 & 0 & 0 & 0 & 0 & 1 & 0\\
0 & 0 & 0 & \sin{\varepsilon_{3}} & 0 & 0 &\cos{\varepsilon_{3}}
\end{pmatrix}\text{,}  \qquad&
\mathfrak{R}_{4}= &
\begin{pmatrix}
1 & 0 & 0 & {\frac{-\varepsilon_{4}}{2}} & 0 & 0 & 0\\
0 & 1 & 0 & 0 & 0 & 0 & 0\\
0 & 0 & 1 & 0 & 0 & 0 & \varepsilon_{4}\\
0 & 0 & 0 & 1 & 0 & 0 & 0\\
0 & 0 & 0 & 0 & 1 & 0 & 0\\
0 & 0 & 0 & 0 & 0 & 1 & 0\\
0 & 0 & 0 & 0 & 0 & 0 & 1\\
\end{pmatrix}\text{,} 
\end{alignat*}
\begin{alignat*}{3}
\mathfrak{R}_{5}= &
\begin{pmatrix}
1 & 0 & 0 & 0 & {\frac{\varepsilon_{5}}{2}} & 0 & 0\\
0 & 1 & 0 & 0 & 0 & \varepsilon_{5} & 0 \\
0 & 0 & 1 & 0 & 0 & 0 & 0\\
0 & 0 & 0 & 1 & 0 & 0 & 0\\
0 & 0 & 0 & 0 & 1 & 0 & 0\\
0 & 0 & 0 & 0 & 0 & 1 & 0\\
0 & 0 & 0 & 0 & 0 & 0 & 1\\
\end{pmatrix}\text{,}  \qquad&
\mathfrak{R}_{6}= &
\begin{pmatrix}
1 & 0 & 0 & 0 & 0 & {\frac{-\varepsilon_{6}}{2}} & 0\\
0 & 1 & 0 & 0 & 0 & 0 & 0\\
0 & 0 & 1 & 0 & 0 & 0 & 0\\
0 & 0 & 0 & 1 & 0 & 0 & 0\\
0 & 0 & 0 & 0 & 1 & 0 & 0\\
0 & 0 & 0 & 0 & 0 & 1 & 0\\
0 & 0 & 0 & 0 & 0 & 0 & 1\\
\end{pmatrix}\text{,}  \qquad&
\mathfrak{R}_{7}= &
\begin{pmatrix}
1 & 0 & 0 & 0 & 0 & 0 & {\frac{-\varepsilon_{7}}{2}} \\
0 & 1 & 0 & 0 & 0 & 0 & 0 \\
0 & 0 & 1 & -\varepsilon_{7} & 0 & 0 & 0\\
0 & 0 & 0 & 1 & 0 & 0 & 0\\
0 & 0 & 0 & 0 & 1 & 0 & 0\\
0 & 0 & 0 & 0 & 0 & 1 & 0\\
0 & 0 & 0 & 0 & 0 & 0 & 1\\
 \end{pmatrix}\text{.} 
\end{alignat*}
Therefore, the global adjoint transformation matrix $\mathcal{A}=\mathfrak{R}_{7}\mathfrak{R}_{6}\mathfrak{R}_{5}\mathfrak{R}_{4}\mathfrak{R}_{3}\mathfrak{R}_{2}\mathfrak{R}_{1}$ is generated as:
$$
\mathcal{A}=\left[\begin{array}{ccccccc}
1 & -\varepsilon_{2} e^{\varepsilon_{1}} & 0 & (\frac{-\varepsilon_{4} cos\varepsilon_{3}-\varepsilon_{7} sin \varepsilon_{3}}{2}) e^\frac{\varepsilon_{1}}{2} & (\frac{\varepsilon_{5}}{2}) e^\frac{-\varepsilon_{1}}{2} & (\frac{-\varepsilon_{5}\varepsilon_{2}-\varepsilon_{6}}{2}) e^\frac{\varepsilon_{1}}{2} & (\frac{\varepsilon_{4} sin\varepsilon_{3}-\varepsilon_{7} cos \varepsilon_{3}}{2}) e^\frac{\varepsilon_{1}}{2} \\
0 & e^{\varepsilon_{1}} & 0 & 0 & 0 & \varepsilon_{5} e^\frac{\varepsilon_{1}}{2} & 0 \\
0 & 0 & 1 & (-\varepsilon_{7} cos\varepsilon_{3}+\varepsilon_{4} sin \varepsilon_{3}) e^\frac{\varepsilon_{1}}{2} & 0 & 0 & (\varepsilon_{4} cos\varepsilon_{3}+\varepsilon_{7} sin \varepsilon_{3}) e^\frac{\varepsilon_{1}}{2} \\
0 & 0 & 0 & \cos{\varepsilon_{3}} \hspace{0.1cm} 
 e^\frac{\varepsilon_{1}}{2} & 0 & 0 & -\sin{\varepsilon_{3}}\hspace{0.1cm}  e^\frac{\varepsilon_{1}}{2}  \\
0 & 0 & 0 & 0 & e^\frac{-\varepsilon_{1}}{2} & -\varepsilon_{2} e^\frac{\varepsilon_{1}}{2} & 0 \\
0 & 0 & 0 & 0 & 0 & e^\frac{\varepsilon_{1}}{2} & 0 \\
0 & 0 & 0 & sin\varepsilon_{3}\hspace{0.1cm} 
 e^\frac{\varepsilon_{1}}{2} & 0 & 0 & cos\varepsilon_{3} \hspace{0.1cm} e^\frac{\varepsilon_{1}}{2} \\
\end{array}\right]\text{.} 
$$
\renewcommand{\arraystretch}{1.4}
\begin{table}[h]
    \centering
    \begin{tabular}{l|lllllll}
    \hline
        $\operatorname{Ad}(\exp{\varepsilon(\star)}\star)$ & $\mathfrak{D}_1$ &$ \mathfrak{D}_2$ & $\mathfrak{D}_3$ & $\mathfrak{D}_4$ & $\mathfrak{D}_5$ & $\mathfrak{D}_6$ & $\mathfrak{D}_7$  \\ \hline
        $\mathfrak{D}_1$ & $\mathfrak{D}_1$ & $e^{\varepsilon}\mathfrak{D}_2$ & $\mathfrak{D}_3$ & $e^\frac{\varepsilon}{2} \mathfrak{D}_4$ & $e^\frac{-\varepsilon}{2} \mathfrak{D}_5$ & $e^\frac{\varepsilon}{2} \mathfrak{D}_6$ & $e^\frac{\varepsilon}{2} \mathfrak{D}_7$ \\ 
    
        $\mathfrak{D}_2$ & $\mathfrak{D}_1-\varepsilon \mathfrak{D}_2$ & $\mathfrak{D}_2$ & $\mathfrak{D}_3$ & $\mathfrak{D}_4$ & $\mathfrak{D}_5-\varepsilon \mathfrak{D}_6$ & $\mathfrak{D}_6$ & $\mathfrak{D}_7$  \\ 
        $\mathfrak{D}_3$ & $\mathfrak{D}_1$ & $\mathfrak{D}_2$ & $\mathfrak{D}_3$ & $\mathfrak{D}_4 \cos{\varepsilon}- \mathfrak{D}_7 \sin{\varepsilon}$ & $\mathfrak{D}_5$ & $\mathfrak{D}_6$ & $\mathfrak{D}_4 \sin{\varepsilon}+ \mathfrak{D}_7 \cos{\varepsilon}$ \\
        $\mathfrak{D}_4$ & $\mathfrak{D}_1-\frac{\varepsilon}{2} \mathfrak{D}_4$ & $\mathfrak{D}_2$ & $\mathfrak{D}_3+\varepsilon \mathfrak{D}_7$ & $\mathfrak{D}_4$ & $\mathfrak{D}_5$ & $\mathfrak{D}_6$ & $\mathfrak{D}_7$ \\ 
        $\mathfrak{D}_5$ & $\mathfrak{D}_1+\frac{\varepsilon}{2} \mathfrak{D}_5$ & $\mathfrak{D}_2+ \varepsilon \mathfrak{D}_6$ & $\mathfrak{D}_3$ & $\mathfrak{D}_4$ & $\mathfrak{D}_5$ & $\mathfrak{D}_6$ & $\mathfrak{D}_7$ \\ 
        $\mathfrak{D}_6$ & $\mathfrak{D}_1-\frac{\varepsilon}{2} \mathfrak{D}_6$ & $\mathfrak{D}_2$ & $\mathfrak{D}_3$ & $\mathfrak{D}_4$ & $\mathfrak{D}_5$ & $\mathfrak{D}_6$ & $\mathfrak{D}_7$ \\ 
        $\mathfrak{D}_7$ & $\mathfrak{D}_1-\frac{\varepsilon}{2} \mathfrak{D}_7$ & $\mathfrak{D}_2$ & $\mathfrak{D}_3-\varepsilon \mathfrak{D}_4$ & $\mathfrak{D}_4$ & $\mathfrak{D}_5$ & $\mathfrak{D}_6$ & $\mathfrak{D}_7$ \\ \hline
        \end{tabular}
    \caption{Adjoint Representation Table.}
\end{table}
\renewcommand{\arraystretch}{1.4}
\begin{table}[h]
    \centering
    \begin{tabular}{c|ccccccc}\hline
        $\operatorname{Ad}(\exp{(\varepsilon \mathfrak{D}_i)}\mathfrak{D})$ & $\mathfrak{D}_1$ &$\mathfrak{D}_2$ & $\mathfrak{D}_3$ & $\mathfrak{D}_4$ & $\mathfrak{D}_5$ & $\mathfrak{D}_{6}$ & $\mathfrak{D}_{7}$ \\ \hline
        $\operatorname{Ad}(\exp{(\varepsilon \mathfrak{D}_1)}\mathfrak{D})$ & $\mathfrak{l}_1$ & $e^{\varepsilon}\mathfrak{l}_2$ & $\mathfrak{l}_3$ & $e^{\frac{\varepsilon}{2}} \mathfrak{l}_4$ & $e^{\frac{-\varepsilon}{2}} \mathfrak{l}_5$ & $e^{\frac{\varepsilon}{2}} \mathfrak{l}_6$ & $e^{\frac{\varepsilon}{2}} \mathfrak{l}_7$\\ 
        $\operatorname{Ad}(\exp{(\varepsilon \mathfrak{D}_2)}\mathfrak{D})$ & $\mathfrak{l}_1$ & $\mathfrak{l}_2-\varepsilon \mathfrak{l}_1$ & $\mathfrak{l}_3$ & $\mathfrak{l}_4$ & $\mathfrak{l}_5$ & $\mathfrak{l}_6-\varepsilon \mathfrak{l}_5$ & $\mathfrak{l}_7$ \\ 
        $\operatorname{Ad}(\exp{(\varepsilon \mathfrak{D}_3)}\mathfrak{D})$ & $\mathfrak{l}_1$ & $\mathfrak{l}_2$ & $\mathfrak{l}_3$ & $\mathfrak{l}_4\cos{\varepsilon} +\mathfrak{l}_7 \sin{\varepsilon}$ & $\mathfrak{l}_5$ & $\mathfrak{l}_6$ & $\mathfrak{l}_7 \cos{\varepsilon}-\mathfrak{l}_4 \sin{\varepsilon}$ \\ 
        $\operatorname{Ad}(\exp{(\varepsilon \mathfrak{D}_4)}\mathfrak{D})$ & $\mathfrak{l}_1$ & $\mathfrak{l}_2$ & $\mathfrak{l}_3$ & $\mathfrak{l}_4 -\frac{\varepsilon \mathfrak{l}_1}{2}$ & $\mathfrak{l}_5$ & $\mathfrak{l}_6$ & $\mathfrak{l}_7+\varepsilon \mathfrak{l}_3$ \\ 
        $\operatorname{Ad}(\exp{(\varepsilon \mathfrak{D}_5)}\mathfrak{D})$ & $\mathfrak{l}_1$ & $\mathfrak{l}_2$ & $\mathfrak{l}_3$ & $\mathfrak{l}_4$ & $\mathfrak{l}_5+\frac{\varepsilon \mathfrak{l}_1}{2}$ & $\mathfrak{l}_6+\varepsilon \mathfrak{l}_2$ & $ \mathfrak{l}_7$ \\ 
        $\operatorname{Ad}(\exp{(\varepsilon \mathfrak{D}_6)}\mathfrak{D})$ & $\mathfrak{l}_1$ & $\mathfrak{l}_2$ & $\mathfrak{l}_3$ & $\mathfrak{l}_4$ & $\mathfrak{l}_5$ & $\mathfrak{l}_6-\frac{\varepsilon \mathfrak{l}_1}{2}$ & $\mathfrak{l}_7$ \\ 
        $\operatorname{Ad}(\exp{(\varepsilon \mathfrak{D}_7)}\mathfrak{D})$ & $\mathfrak{l}_1$ & $\mathfrak{l}_2$ & $\mathfrak{l}_3$ & $\mathfrak{l}_4 -\varepsilon \mathfrak{l}_3$ & $\mathfrak{l}_5$ & $\mathfrak{l}_6$ & $\mathfrak{l}_7-\frac{\varepsilon \mathfrak{l}_1}{2}$ \\ 
        \hline 
    \end{tabular}
    \caption{Table for Construction of Invariant Functions.}
    \label{t2}
\end{table}
\begin{lemma}
    Within $\mathbb{L}^{7}$, the subalgebra generated by $\mathfrak{D}_{1}$ and $\mathfrak{D}_{3}$ (denoted as $\mathcal{M}=\mathfrak{l}_{1}$ and $\mathcal{N}=\mathfrak{l}_{3}$) remains invariant under the adjoint action.
\end{lemma}
\begin{proof}
The data in Table \ref{t2} demonstrate the evidence.
\end{proof}
The subsequent lemmas and their Ad-invariance properties are summarized in Table \ref{tab:invariants}.
\begin{table}[ht]
    \centering
    \renewcommand{\arraystretch}{1.1} 
    \begin{tabular}{|c|c|c|}
        \hline
        \textbf{Invariant Function} & \textbf{Definition} & \textbf{Invariance Condition} \\
        \hline
        $\mathbb{O}$ & $\mathbb{O} = 1$ if $\mathfrak{l}_1^2 + \mathfrak{l}_2^2 + \mathfrak{l}_3^2 \neq 0$, otherwise $0$ & Preserved under adjoint actions. \\
        $\mathcal{P}$ & $\mathcal{P} = 1$ if $\mathfrak{l}_1^2 + \mathfrak{l}_2^2 + \mathfrak{l}_3^2 + \mathfrak{l}_5^2 + \mathfrak{l}_6^2 \neq 0$, otherwise $0$ & Similar invariance condition as $\mathbb{O}$. \\
        $\mathcal{Q}$ & $\mathcal{Q} = \operatorname{sgn}(\mathfrak{l}_4)$ if $\mathfrak{l}_1 = \mathfrak{l}_3 = \mathfrak{l}_7 = 0$, otherwise $0$ & $\operatorname{sgn}(\mathfrak{l}_4)$ remains unchanged under adjoint actions. \\
        $\mathcal{R}$ & $\mathcal{R} = \operatorname{sgn}(\mathfrak{l}_2)$ if $\mathfrak{l}_1 = 0$, otherwise $0$ & Similar invariance condition as $\mathcal{Q}$. \\
        $\mathcal{S}$ & $\mathcal{S} = \operatorname{sgn}(\mathfrak{l}_5)$ if $\mathfrak{l}_1 = 0$, otherwise $0$ & Similar invariance condition as $\mathcal{Q}$. \\
        $\mathcal{T}$ & $\mathcal{T} = \operatorname{sgn}(\mathfrak{l}_6)$ if $\mathfrak{l}_1 = \mathfrak{l}_2 = \mathfrak{l}_5 = 0$, otherwise $0$ &  Similar invariance condition as $\mathcal{Q}$. \\
        $\mathcal{U}$ & $\mathcal{U} = \operatorname{sgn}(\mathfrak{l}_7)$ if $\mathfrak{l}_1 = \mathfrak{l}_3 = \mathfrak{l}_4 = 0$, otherwise $0$ &  Similar invariance condition as $\mathcal{Q}$. \\
         \hline
    \end{tabular}
    \caption{Invariant Functions and Their Conditions.}
    \label{tab:invariants}
\end{table}
\subsection{Construction of Optimal System}
\quad \quad The subalgebras listed below categorize the one-parameter optimal system for $\mathbb{L}^7$.
\subsection{Case 1.} 
$\begin{aligned}
(i) \hspace{0.2em}\Lambda_{1}=\left\{\mathfrak{l}_{1}=0,
\mathfrak{l}_{3}=1\right\}\text{,} \quad (ii) \hspace{0.2em}\Lambda_{2}=\left\{\mathfrak{l}_{1}=1, \mathfrak{l}_{3}=0\right\} \text{,}\quad (iii) \hspace{0.2em} \Lambda_{3}=\left\{\mathfrak{l}_{1}=0, \mathfrak{l}_{3}=0\right\}\text{.}
\end{aligned}$
\subsection{Case 2. \texorpdfstring{$k \neq 0$}{k ≠ 0}, 
\texorpdfstring{$\Lambda_{4}$}{Lambda4} = \texorpdfstring{$\left\{ \mathfrak{l}_{1}=1, \mathfrak{l}_{3}=k \right\}$}{set notation}}
To create the ideal system of subalgebras, we apply the technique from \cite{hu2015direct}. If at least one $\varepsilon_{i}$, i= 1,…,5 can be solved from the following adjoint transformation equation, then the symbolic elements $\sum_{j=1}^{7} \mathfrak{l}_{i} \mathfrak{D}_{i}$ and $\sum_{j=1}^{7} \mathfrak{l}_{i}^{*}\mathfrak{D}_{j}$ are equivalent, which should be known before examining the aforementioned cases:
$$(\mathfrak{l}_{1}^{*},\mathfrak{l}_{2}^{*},\mathfrak{l}_{3}^{*},\mathfrak{l}_{4}^{*},\mathfrak{l}_{5}^{*},\mathfrak{l}_{6}^{*},\mathfrak{l}_{7}^{*})=(\mathfrak{l}_{1},\mathfrak{l}_{2},\mathfrak{l}_{3},\mathfrak{l}_{4},\mathfrak{l}_{5},\mathfrak{l}_{6},\mathfrak{l}_{7}). \mathcal{A}$$
An ideal system would consist of the symbolic components of the collection of equations:
\begin{equation}\label{1111}
    \begin{split}
     &\mathfrak{l}_{1}^{*}=\mathfrak{l}_{1}, \hspace{0.2em} \mathfrak{l}_{2}^{*}=(-\mathfrak{l}_{1} \varepsilon_{2} +\mathfrak{l}_{2})e^{\varepsilon_{1}}, \hspace{0.2em} \mathfrak{l}_{3}^{*}=\mathfrak{l}_{3}, \hspace{0.2em} \\& \mathfrak{l}_{4}^{*}=[-(\varepsilon_{4}cos\varepsilon_{3}+\varepsilon_{7}sin\varepsilon_{3})\frac{\mathfrak{l}_{1}}{2}+ (-\varepsilon_{7}cos\varepsilon_{3}+\varepsilon_{4}sin\varepsilon_{3})\mathfrak{l}_{3}+\mathfrak{l}_{4}cos\varepsilon_{3}+
    \mathfrak{l}_{7}sin\varepsilon_{3}] e^\frac{\varepsilon_1}{2}, \hspace{0.2em} \mathfrak{l}_{5}^{*}=(\frac{\mathfrak{l}_{1}\varepsilon_{5}}{2}+\mathfrak{l}_{5}) e^\frac{-\varepsilon_{1}}{2}, \hspace{0.2em}\\
     &\mathfrak{l}_{6}^{*}=[\frac{-\mathfrak{l}_{1}}{2}(\varepsilon_{5}\varepsilon_{2}+\varepsilon_{6})+\mathfrak{l}_{2}\varepsilon_{5}-\mathfrak{l}_{5}\varepsilon_{2}+\mathfrak{l}_{6}] e^\frac{\varepsilon_{1}}{2},\hspace{0.2em} \\
     &\mathfrak{l}_{7}^{*}=[\frac{\mathfrak{l}_{1}}{2}(\varepsilon_{4}sin\varepsilon_{3}-\varepsilon_{7}cos\varepsilon_{3})+\mathfrak{l}_{3}(\varepsilon_{7}sin\varepsilon_{3}+\varepsilon_{4}cos\varepsilon_{3})+\mathfrak{l}_{4}sin\varepsilon_{3}+\mathfrak{l}_{7}cos\varepsilon_{3}]e^\frac{\varepsilon_{1}}{2},\end{split}
\end{equation}
     has a solution with respect to $\varepsilon_{i}$, for i=1,2,3,4,5,6,7.
   Now for $\Lambda_{1}$: $\mathfrak{l}_{1}=0$ , $\mathfrak{l}_{3}=1$, the representative element is $
\hat{\mathcal{X}}=\mathfrak{l}_{2} \mathfrak{D}_{2}+\mathfrak{D}_{3}+\mathfrak{l}_{4} \mathfrak{D}_{4}+\mathfrak{l}_{5} \mathfrak{D}_{5}+\mathfrak{l}_{6} \mathfrak{D}_{6}+\mathfrak{l}_{7} \mathfrak{D}_{7}\text{,}
$. The adjoint actions of $\mathfrak{D}_{4}$ and $\mathfrak{D}_{7}$ allow us to omit the coefficients of $\mathfrak{D}_{4}$ and $\mathfrak{D}_{7}$, yielding $\hat{\mathcal{X}}=b_{2} \mathfrak{D}_{2}+\mathfrak{D}_{3}+b_{5} \mathfrak{D}_{5}+b_{6} \mathfrak{D}_{6}\text{.}$ The solutions to the system of Eqs. \eqref{1111}, as $\varepsilon_{4}=-\mathfrak{l}_{7}\text{,} \hspace{0.2cm} \varepsilon_{7}=\mathfrak{l}_{4}\text{.}$
The same process may be used to create a one-dimensional optimal subalgebra for the remaining situations, and the results are expressed as:
\begin{thm}
The vector fields generate the optimal one-dimensional subalgebra of $\mathbb{L}^7$ of \eqref{11}.
$$
\begin{aligned}
&\mathcal{X}_{1,b_2,b_5,b_6}=b_2 \mathfrak{D}_{2}+\mathfrak{D}_{3}+b_5 \mathfrak{D}_{5}+b_6 \mathfrak{D}_{6}\text{,}\quad
\mathcal{X}_{2}=\mathfrak{D}_{1}\text{,}\quad
\mathcal{X}_{3, d_2, d_4, d_5, d_6, d_7}=d_{2} \mathfrak{D}_{2}+d_{4} \mathfrak{D}_{4}+d{5} \mathfrak{D}_{5}+d{6} \mathfrak{D}_{6}+d{7} \mathfrak{D}_{7}\text{,}\quad \\
&\mathcal{X}_{4,k}=\mathfrak{D}_{1}+k \mathfrak{D}_{3}\text{,}
\end{aligned}
$$
where $b_2,b_5,b_6$ are parameters with $b_2,b_5,b_6 \in \{-1,0,1\}$. Also, $k$ is to be determined according to earlier mentioned in case $1$, and $d{j}^{'}s$ are arbitrary constants $\forall \hspace{0.1cm} j=2,4,5,6,7 \text{.}$
\end{thm}
\begin{proof}
It is enough to prove the theorem by illustrating that the $\mathcal{X}_i^{'} s$ are distinct from one another. For this, we construct a table demonstrating the invariants in Table \ref{t4}. Analysis of the contents of Table \ref{t4} shows that the subalgebras are distinct.
\renewcommand{\arraystretch}{1.0}
\begin{table}[H]
    \centering
    \begin{tabular}{|c|cccccccccc|}
    \hline
      &$ \mathcal{K}$ & $\mathcal{M}$ & $\mathcal{N}$ & $\mathcal{O}$ & $\mathcal{P}$ & $\mathcal{Q}$ & $\mathcal{R}$ & $\mathcal{S}$ & $\mathcal{T}$ & $\mathcal{U}$\\\hline
       
       $\mathcal{X}_{1,1,0,1}(\mathfrak{D}_2+\mathfrak{D}_3+\mathfrak{D}_6)$  & $-2$ & 0 & 1 & 1 & 1 & 0 & $1$ & $0$ &  0 & 0\\
        $\mathcal{X}_{1,0,0,0}(\mathfrak{D}_3)$  & $-2$ & 0 & 1 & 1 & 1 & 0 & $0$ & $0$ &  $0$ & 0\\
        $\mathcal{X}_{2}(\mathfrak{D}_1)$ & 2  & 1 & 0 & 1 & 1 & 0 & 0 & 0 & 0 & 0\\
       $\mathcal{X}_{3, 0, d_4, d_5, d_6, d_7; d_4, d_5, d_6, d_7 \neq 0}(d_4 \mathfrak{D}_4+ d_5 \mathfrak{D}_5+ d_6 \mathfrak{D}_6 + d_7 \mathfrak{D}_7)$  & 0 & 0 & 0 & 0 & 1 & 0 & $0$ & $d_5$ & 0 & 0\\
      $\mathcal{X}_{3, d_2, 0, d_5, d_6, d_7; d_2, d_5, d_6, d_7 \neq 0}(d_2 \mathfrak{D}_2+ d_5 \mathfrak{D}_5+ d_6 \mathfrak{D}_6 + d_7 \mathfrak{D}_7)$  & 0 & 0 & 0 & 1 & 1 & 0 & $d_2$ & $d_5$ & 0 & $d_7$\\
        $\mathcal{X}_{3, 0, d_4, 0, d_6, d_7; d_4,  d_6, d_7 \neq 0}(d_4 \mathfrak{D}_4+  d_6 \mathfrak{D}_6 + d_7 \mathfrak{D}_7)$  & 0 & 0 & 0 & 0 & 1 & 0 & $0$ & $0$ & $d_6$ & 0\\ 
        $\mathcal{X}_{3, d_2, 0, d_5, d_6, 0; d_2,  d_5, d_6 \neq 0}(d_2 \mathfrak{D}_2+  d_5 \mathfrak{D}_5 + d_6 \mathfrak{D}_6)$  & 0 & 0 & 0 & 1 & 1 & $0$ & $d_2$ & $d_5$ & 0 & $0$\\ 
        $\mathcal{X}_{3, 0, 0, 0, d_6, d_7; d_6, d_7 \neq 0}(d_6 \mathfrak{D}_6 + d_7 \mathfrak{D}_7)$  & 0 & 0 & 0 & 0 & 1 & 0 & $0$ & $0$ & $d_6$ & $d_7$\\
        $\mathcal{X}_{4,1}(\mathfrak{D}_1+\mathfrak{D}_3)$  &  0 & 1 & 1 & 1 & 1 & 0 & 0 & 0 & 0 & 0\\
        $\mathcal{X}_{4,-1}(\mathfrak{D}_1-\mathfrak{D}_3)$ & 0 & 1 & $-1$ & 1 & 1 & 0 & 0 & 0 & 0 & 0\\\hline
    \end{tabular}
    \caption{Value of Invariant Function.}
    \label{t4}
\end{table}
\end{proof}
\section{Symmetry Reduction and Group Invariant Solutions for ZK Equation} \label{5}
\quad \quad Constructing a one-dimensional optimal system simplifies finding group-invariant solutions by leveraging key symmetries, aiding classification, and revealing deeper insights into the equation's structure. 
\subsection{Invariant Solution Using Optimal System}
\subsubsection{\texorpdfstring{$\mathcal{X}=b_{2} \mathfrak{D}_{2}+\mathfrak{D}_{3}+b_{5} \mathfrak{D}_{5}+b_6 \mathfrak{D}_{6}$ }{X = b2 {D}2 + {D}3 +b5{D}5 + b6 {D}6}} \label{S1}
\begin{equation} \label{a}
P=\frac{2xb_2-b_5 t^2-2b_6t}{2b_2}\text{,}\quad
Q=y \sin \left(\frac{t}{b_2}\right)+z \cos \left(\frac{t}{b_2}\right)\text{,} \quad
R=y \cos \left(\frac{t}{b_2}\right)-z \sin \left(\frac{t}{b_2}\right)\text{,}\quad
\text{and}\hspace{0.2cm} u=\frac{b_5 t}{ab_2}+F(P,Q,R)\text{.}
\end{equation}
Here, $P$, $Q$, and $R$  are the similarity variables,  and $F(P, Q, R)$ is the similarity function.
Therefore, Eq. \eqref{11} is converted into
\begin{equation}\label{s}
a(aF{}b_2-b_6)F_{P}-aPF_{R}+aRF_{Q}+ab_2(F_{P P}+F_{Q Q}+F_{R R})+b_5=0\text{.}
\end{equation}
To reduce Eq. \eqref{s}, the listed infinitesimal generators using Lie symmetry are:
\begin{equation} \label{524}
\xi_{P}=0  ,\quad  \xi_{Q}=A_{1}  , \quad \xi_{R}=0  , \quad \xi_{F}=0\text{,}
\end{equation}
where $A_{1}$ is an arbitrary constant,
which gives $F(P,Q,R)=G(r,s )$ with $P=r $, $R=s$ and
$G(r,s)$ is a similarity function of $r$ and $s$.\\
Thus, the second reduction of the model \eqref{11} is:
\begin{equation}\label{1}
a(aG{}b_2-b_6)G_{r}-arG_{s}+ab_2(G_{r r}+G_{s s})+b_5=0\text{.}
\end{equation}
To reduce Eq. \eqref{1}, the listed infinitesimal generators using Lie symmetry are:
\begin{equation} \label{527}
\xi_{r}=0, \quad \xi_{s}=B_{1}, \quad  \xi_{G}=0\text{,} 
\end{equation}
where $B_{1}$ is the arbitrary constant,
which gives $G(r, s)=H(\lambda)$ with $\lambda=r$, and $H(\lambda)$ is a similarity function of $\lambda$.\\
Thus, the third reduction of the model \eqref{11} is:
\begin{equation}\label{5.21}
ab_2H^{\prime \prime}+a(aH{}b_2-b_6)H^{\prime}+b_5=0\text{.}
\end{equation}
The solution to the above ODE for $a=1, b_2=1, b_5=0$ and $b_6=1$ is given by
\begin{equation}
    H(\lambda)=\frac{2}{1+2C e^{-\lambda}}\text{.}
\end{equation}
So, the invariant solution for the model equation. \eqref{11} is as follows: 
\begin{equation}\label{535}
u_1(x,y,z,t)  = \frac{2}{1+2C e^{t-x}}\text{.} \hspace{1cm} 
\end{equation}
\subsubsection{\texorpdfstring{$\mathcal{X}=\mathcal{D}_{1}$}{X = X1}}\label{S2}
\begin{table}[H]
    \centering
    \begin{tabular}{>{\centering\arraybackslash}m{1.5cm}>{\centering\arraybackslash}m{4cm} >{\centering\arraybackslash}m{9.5cm}}
    \toprule
 \textbf{Reduction} &\textbf{Similarity variables} & \textbf{Reduced equations} \\
    \midrule
     
    1.&$P=\frac{x}{\sqrt{t}}$, \;$Q=\frac{y}{\sqrt{t}}$\text{,}\newline $R=\frac{z}{\sqrt{t}}$, \; $u =\frac{F(P,Q,R)}{\sqrt{t}}$&
    $2(F_{P P}+F_{Q Q}+F_{R R}) +2aF{}F_{P}-F{}-(P F_{P}+Q F_{Q}+R F_{R})=0$ \\
    \midrule
    2.&$r = P, \; s = Q^2+R^2$, \newline  $F =G(r, s)$& 
    $aG_{}G_{r} + G_{rr} +4G_s +4sG_{ss}=0$ \\
    
     \bottomrule
    \label{T3}
    \end{tabular}
\end{table}
The solution to the above PDE is given by
\begin{equation}
    G(r,s)=\frac{r}{a}.
\end{equation}
So, the invariant solution for the model equation. \eqref{11} is as follows: 
\begin{equation}\label{542}
u_2(x,y,z,t)  =\frac{x}{a\sqrt{t}}\text{.}
\end{equation}
\subsubsection{\texorpdfstring{$\mathcal{X}=d_{2} \mathcal{D}_{2}+d_{4} \mathcal{D}_{4}+d_{5} \mathcal{D}_{5}+d_{6} \mathcal{D}_{6}+d_{7} \mathcal{D}_{7}$}{X = Î±2 X2 + Î±4 X4 + Î±5 X5}}\label{S3}
\begin{table}[H]
    \centering
    \begin{tabular}{>{\centering\arraybackslash}m{1.5cm}>{\centering\arraybackslash}m{4cm} >{\centering\arraybackslash}m{9.9cm}}
    \toprule
 \textbf{Reduction} &\textbf{Similarity variables} & \textbf{Reduced equations} \\
    \midrule
     
    1.&$P=\frac{2xd_2-d_5 t^2-2d_6t}{2d_2}$,\quad $Q=y-\frac{d_4 t}{d_2}$\text{,} \quad $R=z-\frac{d_7 t}{d_2}$, \; $u =\frac{d_5 t}{ad_2}+F(P,Q,R)$&
    $a(aF{}d_2-d_6)F_{P}-ad_7F_{R}-ad_4F_{Q}+ad_2(F_{P P}+F_{Q Q}+F_{R R})+d_5=0$ \\
    \midrule
    2.&$r = b_1P-b_2Q$, \; $s = b_3Q-b_1R $ ,\newline  $F =G(r, s)$&  
    $a(ab_1d_2G{}-b_2d_4)G_{r}+ ad_2(b_1^2+b_2^2)G_{r r}+ ad_2(b_1^2+b_3^2)G_{s s}+ a(b_1d_7+b_3d_4)G_{s}- 2ad_2b_2b_3G_{r s}+ d_5=0$  \\
     \midrule
    3.&$\lambda=A_2 r- A_1 s$,\newline  $G =H(\lambda)$&  $a((b_1^2+b_2^2)A_2^2+2A_1A_2b_2b_3+A_1^2(b_1^2+b_3^2))d_2 H^{\prime \prime} +a(aH{}A_2b_1d_2+-b_2d_4A_2+A_1(-b_1d_7-b_3d_4)) H^{\prime}+d_5=0$ \\
     \bottomrule
    \label{T4}
    \end{tabular}
\end{table}
The solution to the above ODE for $b_1=0, b_2=1, b_3=1, A_1=1$ and $A_2=1$ is given by
\begin{equation}
H(\lambda)  =\left[\frac{d_5\lambda}{2d_4a}\right] \text{.}
\end{equation}
So, the invariant solution for the model equation. \eqref{11} is as follows: 
\begin{equation}\label{55}
u_3(x, y, z, t)  =\frac{d_5}{d_4a}({\frac{d_4t}{d_2}}-y ) \text{.}
\end{equation}
\subsubsection{\texorpdfstring{$\mathcal{X}=\mathcal{D}_{1}+k \mathcal{D}_{3} \hspace{0.1cm} \text{,} \hspace{0.1cm} k\neq 0$}{X = X1 + k X3 ; k != 0}}\label{S4}
\begin{table}[H]
    \centering
    \begin{tabular}{>{\centering\arraybackslash}m{1.5cm}>{\centering\arraybackslash}m{4cm} >{\centering\arraybackslash}m{9.5cm}}
    \toprule
 \textbf{Reduction} &\textbf{Similarity variables} & \textbf{Reduced equations} \\
    \midrule
     
    1.&$P=\frac{x}{\sqrt{t}}$,\quad $Q=\frac{[z \cos (k \ln (t))+y \sin (k \ln (t))]}{\sqrt{t}}$\text{,} \quad $R=\frac{[-z \sin (k \ln (t))+y \cos (k \ln (t))]}{\sqrt{t}}$, \; $u =\frac{F(P,Q,R)}{\sqrt{t}}$&
    $(kR-\frac{Q}{2})F_{Q}-(kQ+\frac{R}{2})F_{R}+aF{}F_{P}-{\frac{P}{2}}F_{P}-\frac{F{}}{2}+F_{P P}+F_{Q Q}+F_{R R}=0$\\
    \midrule
    2.&$r = P, \; s = Q^2+R^2$ ,\newline  $F =G(r, s)$& 
    $G_{r r}+4sG_{s s}+aGG_{r}+4G_{r}-sG_{s}-{\frac{r}{2}}G_{r}-\frac{G}{2}=0$  \\
     \bottomrule
    \label{T5}
    \end{tabular}
\end{table}
The solution to the above PDE is given by
\begin{equation}
    G(r,s)=\frac{r-8}{a}.
\end{equation}
So, the invariant solution for the model equation. \eqref{11} is as follows: 
\begin{equation}\label{5421}
u_4(x,y,z,t)  =\frac{x-8\sqrt{t}}{a\sqrt{t}}\text{.}
\end{equation}
 \subsection{Invariant Solution Using Analysis of Parameters}
\begin{equation}
    \frac{dx}{\xi_x}= \frac{dy}{\xi_y}=\frac{dz}{\xi_z}= \frac{dt}{\xi_t}= \frac{du}{\xi_u}\text{.}
\end{equation}
\subsubsection{\texorpdfstring{Reduction using Translation symmetry $\mathcal{D}_2$}{X2}} \label{S5}
\begin{equation}
\begin{split}
& P={x},\quad Q={y},\quad R={z},\
\text{and} \hspace{0.1cm} u=F(P,Q,R)\text{,}
\end{split}
\end{equation}
Here, $P$, $Q$, and $R$ are the similarity variables, and $F(P, Q, R)$ is the similarity function. Therefore, Eq. \eqref{11} is converted into
\begin{equation}\label{rpa3}
\begin{split}
& aF{}F_{P}+F_{P P}+F_{Q Q}+F_{R R}=0\text{.} 
\end{split}
\end{equation}
To reduce Eq. \eqref{rpa3}, the listed infinitesimal generators using Lie symmetry are as follows:
\begin{equation}\label{567}
\xi_{P}=d_{1} P+ d_{4}, \quad \xi_{Q}=d_{1} Q+ d_{2} R+ d_{3}, \quad \xi_{R}=d_{1} R- d_{2} Q+ d_{5}, \quad \xi_{F}=-d_{1} F\text{,}
\end{equation}
where $d_{1}$, $d_{2}$, $d_{3}$, $d_{4}$, $d_{5}$  are the arbitrary constants.The characteristic equations of  Eq. \eqref{567} are:
\begin{equation}
\frac{d P}{d_{1} P+ d_{4}}=\frac{d Q}{d_{1} Q+ d_{2} R+ d_{3}}=\frac{d R}{d_{1} R-d_{2} Q+ d_{5}}=\frac{d F}{-d_{1} F}\text{.}
\end{equation} 
\begin{table}[H]
    \centering
    \begin{tabular}{>{\centering\arraybackslash}m{2cm} >{\centering\arraybackslash}m{4cm} >{\centering\arraybackslash}m{5.5cm} >{\centering\arraybackslash}m{5cm}}
    \toprule
    \textbf{Subcase} & \textbf{Similarity variables} & \textbf{Reduced equations} &\textbf{Invariant solutions}  \\
    \midrule
     $d_1 \neq 0$ & 
    $r = \frac{P}{Q}, \; s = \frac{R}{Q},$ \newline $F =\frac{G(r, s)}{Q}$ \vspace{1em} 
    \newline $\lambda=\frac{r^2}{1+s^2},$  \newline $G=\frac{H(\lambda)}{r}$ & 
    $(r^2+1)G_{r r}+(s^2+1)G_{s s}+aG{}G_{r}+4rG_{r}+2G{}+2rsG_{r s}+4sG_{s}=0$ \vspace{2em} 
    \newline $-aH^2+4\lambda^2(\lambda+1) H^{\prime \prime}+2\lambda(aH+2\lambda-1) H^{\prime}+ 2H=0$ &
    $u_5(x,y,z,t)=\frac{2x}{2C\sqrt{x^2+y^2+z^2}\hspace{0.1cm} +ax}$\\
    \midrule
    $d_2 \neq 0$ & 
    $r = P, \; s = Q^2+R^2,$ \newline $F =G(r, s)$ \vspace{1em} 
    \newline $\lambda=\frac{rA_{1}+A_{2}}{A_{1}\sqrt{s}},$  \newline $G=\frac{H(\lambda)}{\sqrt{s}}$ & 
    $aG_{}G_{r} + G_{rr} +4G_s +4sG_{ss}=0$ \vspace{2em} 
    \newline $(A_{1} \lambda^2+A_{1}) H^{\prime \prime}+(aA_{1} H +3A_{1}\lambda) H^{\prime}+ H=0$  &
    No solution \\
    \midrule
   $d_3 \neq 0$ & 
    $r = P, \; s = R,$ \newline $F = G(r, s)$ \vspace{1em} 
    \newline $\lambda=\frac{A_{1} r+A_{2}}{A_{1}(A_{1} s+A_{3})},$  \newline $G=\frac{H(\lambda)}{A_1s+A_{3}}$ & 
    $AGG_r + G_{rr} + G_{ss} = 0$  \vspace{2em} 
    \newline $ (1+ \lambda^2A_{1}^2) H^{\prime \prime}+(aH +4A_{1}^2\lambda) H^{\prime}+ 2A_{1}^2H=0$ &
    $u_6(x,y,z,t)=\frac{4A_1}{4CA_1(A_1 B+1)+(a B^2+a) \tan^{-1}B+ a B}$ \vspace{1em}
    \newline $B=\frac{A_1x+A_2}{A_1y+A_3}$\\
 \bottomrule
    \label{T6}
    \end{tabular}
\end{table} 
 \section{Results and Discussion} \label{6}
 The physical behavior of the model \eqref{11} is shown in this section using the solutions from the previous section, and the following observations have been made:
 \begin{enumerate} 
\item 
 The solution of equation \eqref{535} is illustrated in Figure 1. Figure 1(a) represents the evolution of wave amplitude, which decreases gradually. Figure 1(b) represents the contour sketch of the equation \eqref{535} that highlights concentric level curves. This solution could resemble a stable wave structure with controlled dispersion.

\item The solution of equation \eqref{542} is illustrated in Figure 2. Figure 2(a) shows that the amplitude $u(x, t)$ is changing as the time increases. This directional change in the amplitude $u(x, t)$ represents wave propagation influenced by the magnetic field. It shows sharper variation as compared to $u_1(x,t)$. Figure 2(b) highlights directional dependence in the propagation of the wave, which directs the result that $u_2(x,t)$ has stronger dispersive features with local steepening in amplitude.

\item The solution of equation \eqref{55} is illustrated in Figure 3. Figure 3(a) shows the linear variation of $u(y,t)$. The yellow and green color gradient highlights the negative amplitude, whereas blue represents the positive amplitude. Wave amplitude decreases linearly with $y$ and $t$ as it is tilted diagonally. Figure 3(b) represents the contour sketch of Eq. \eqref{55}. This type of wave is observed in magnetized plasma, such as in Earth's magnetosphere.

\item The solution of equation \eqref{5421} is illustrated in Figure 4. The solution profile $u_4(x,t)$ is similar to that of $u_2(x,t)$, as both show smooth growth without localized peaks. Similarly, the contour plot in Figure 4(b) displays steady amplitude variation. 

\item Figure 5(a) is the $3D$ graph for the subcase $d_1\neq 0$. The surface plot of $u_5(x,y)$ shows that it has less amplitude than $u_4(x,t)$. Figure 5(b) represents the contour plot for the subcase $d_1\neq 0$, which shows how the intensity of the wave varies gradually. This behavior of $u_5(x,y)$ shows its weak dispersion effects.

 \item Figure 6(a) represents the $3D$ graph for the subcase $d_3\neq 0$. The solution profile of $u_6(x,y)$ displays a linear pattern of growth with a little bit of variation of $x$ and $y$. Figure 6(b) represents the contour sketch for the subcase $d_3 \neq 0$. This shows that $u_6(x,y)$ doesn't have strong oscillatory behavior, which directs its steadily amplifying solution.
\end{enumerate}
\begin{figure}[H]
\centering
\begin{minipage}{.38\textwidth}
    \subfloat[]{\includegraphics[width=\textwidth]{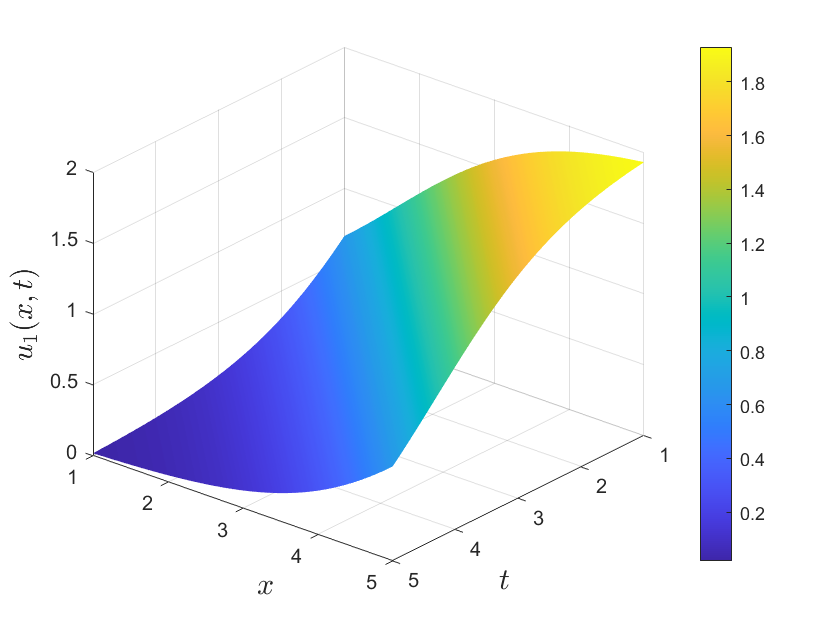}\label{fig:sub_1a}}
\end{minipage}
\begin{minipage}{.38\textwidth}
    \subfloat[]{\includegraphics[width=\textwidth]{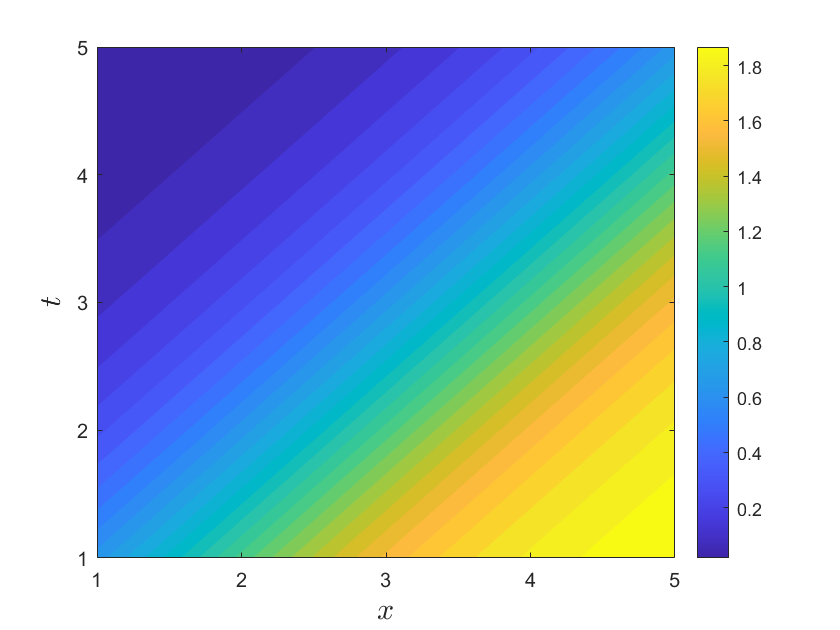}\label{fig:sub_1d}}
\end{minipage}\\
    \caption{Solution profile of amplitude $u_1(x,t)$ to Eq. \ref{535}.}\label{fig:1}
    \begin{minipage}{.38\textwidth}
    \subfloat[]{\includegraphics[width=\textwidth]{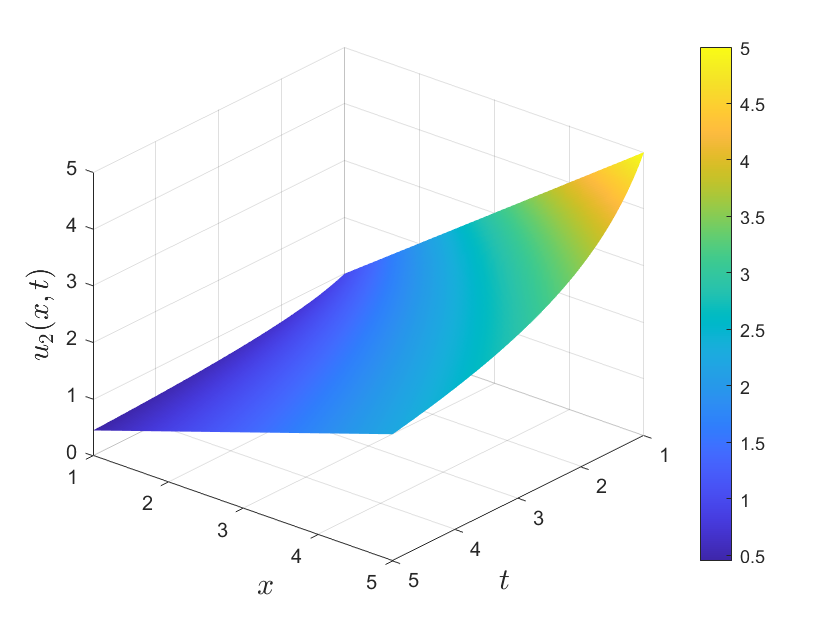}\label{fig:sub_2aa}}
\end{minipage}
   \begin{minipage}{.38\textwidth}
    \subfloat[]{\includegraphics[width=\textwidth]{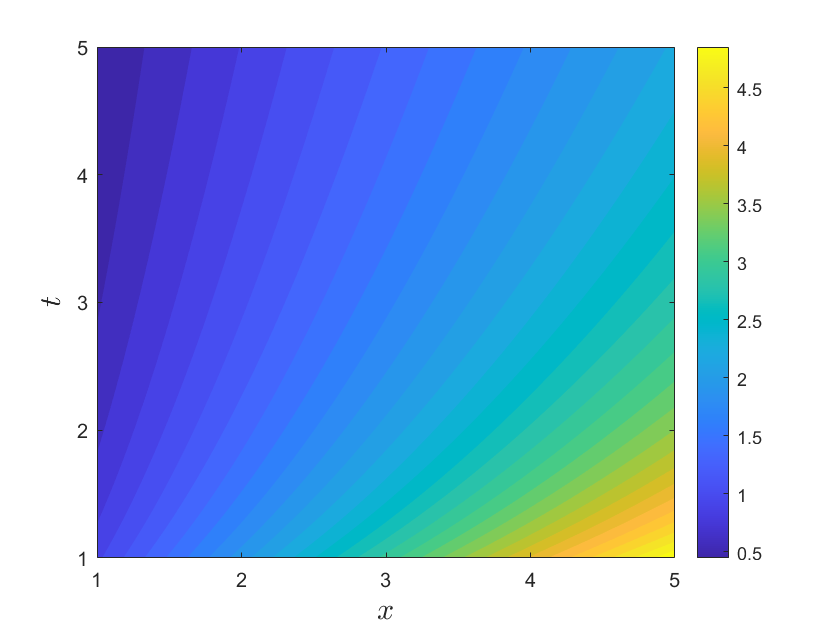}\label{fig:sub_2d}}
\end{minipage}\\
    \caption{Solution profile of amplitude $u_2(x,t)$ to Eq. \ref{542}.}\label{fig:2}
    \begin{minipage}{.38\textwidth}
    \subfloat[]{\includegraphics[width=\textwidth]{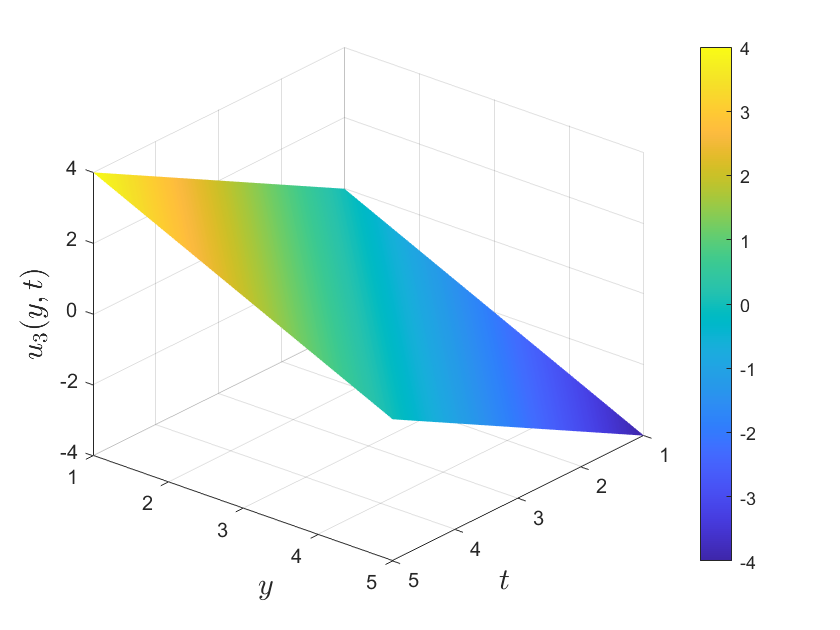}\label{fig:sub_3a}}
\end{minipage}
 \begin{minipage}{.38\textwidth}
    \subfloat[]{\includegraphics[width=\textwidth]{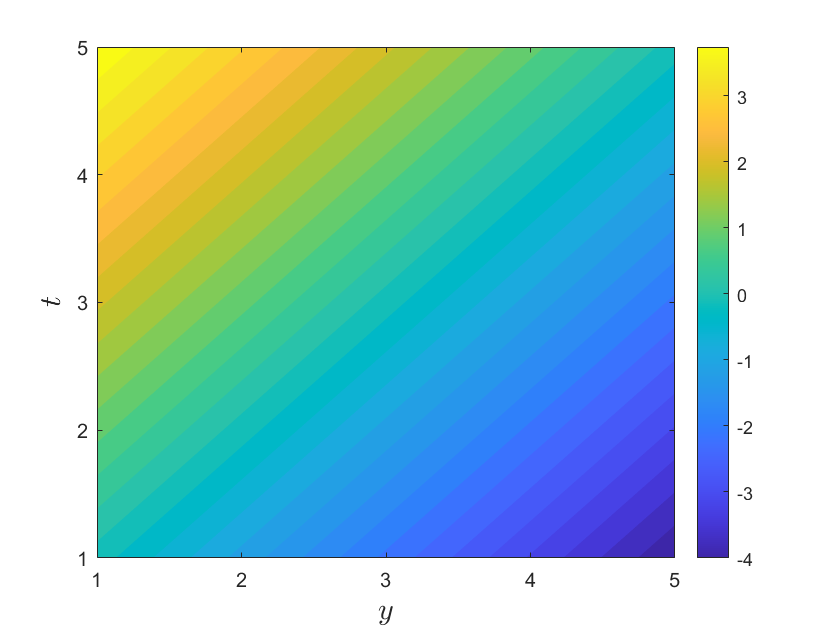}\label{fig:sub_3b}}
\end{minipage}\\
    \caption{Solution profile of amplitude $u_3(y,t)$ to Eq. \ref{55}.}\label{fig:3}
     \end{figure}
    \begin{figure}[H]
\centering
     \begin{minipage}{.38\textwidth}
    \subfloat[]{\includegraphics[width=\textwidth]{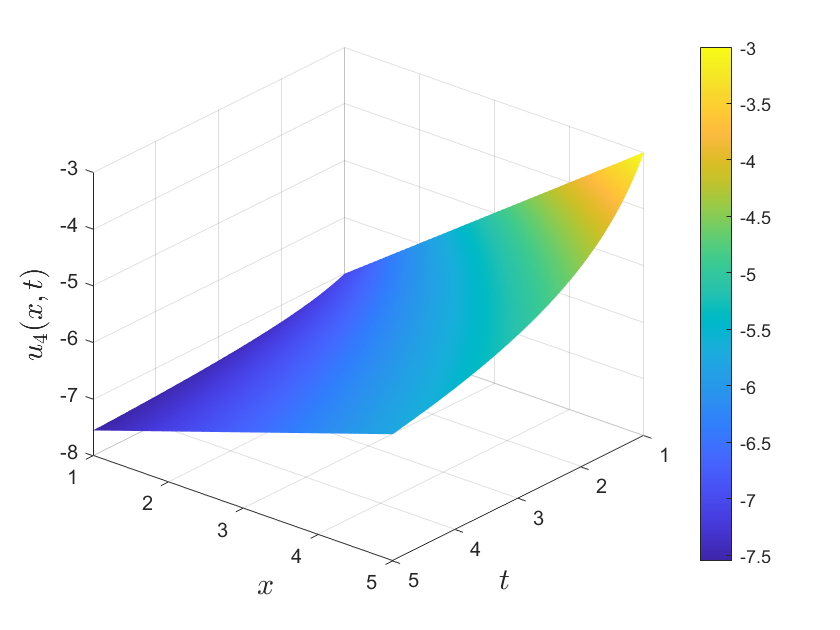}\label{fig:sub_2a}}
\end{minipage}
   \begin{minipage}{.38\textwidth}
    \subfloat[]{\includegraphics[width=\textwidth]{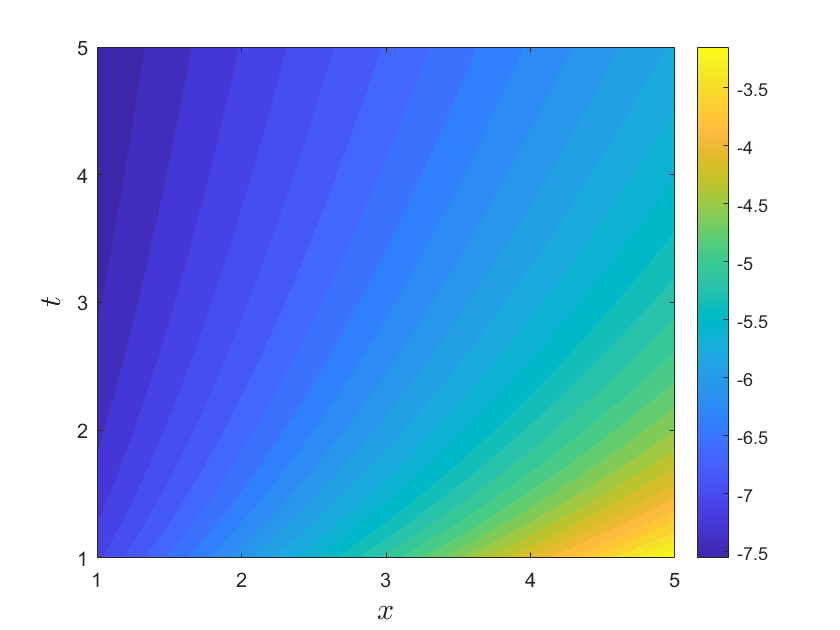}\label{fig:sub_2da}}
\end{minipage}\\
    \caption{Solution profile of amplitude $u_4(x,t)$ to Eq. \ref{5421}.}\label{fig:3aa}
    \end{figure}
    \begin{figure}[H]
\centering
\begin{minipage}{.38\textwidth}
 \subfloat[]{\includegraphics[width=\textwidth]{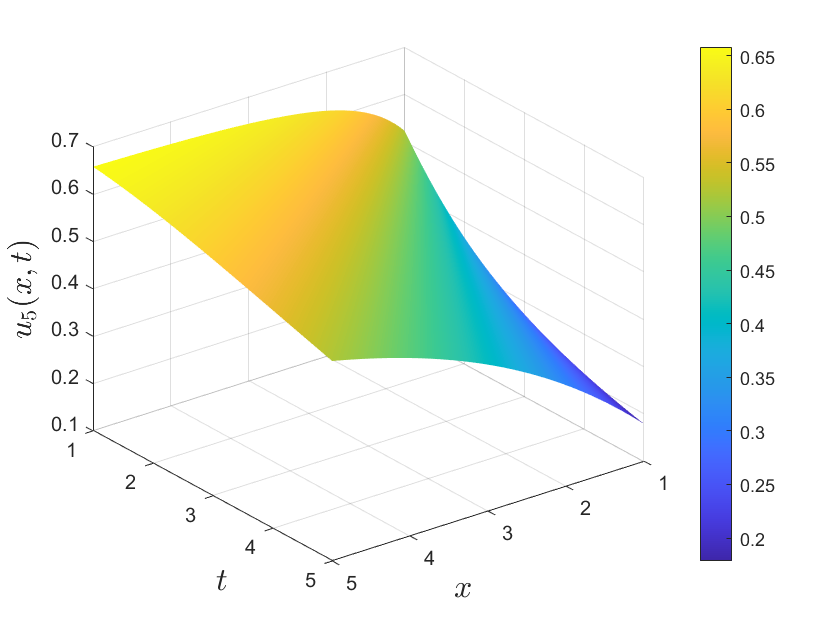}\label{fig:sub_4a}}
\end{minipage}
\begin{minipage}{.38\textwidth}
    \subfloat[]{\includegraphics[width=\textwidth]{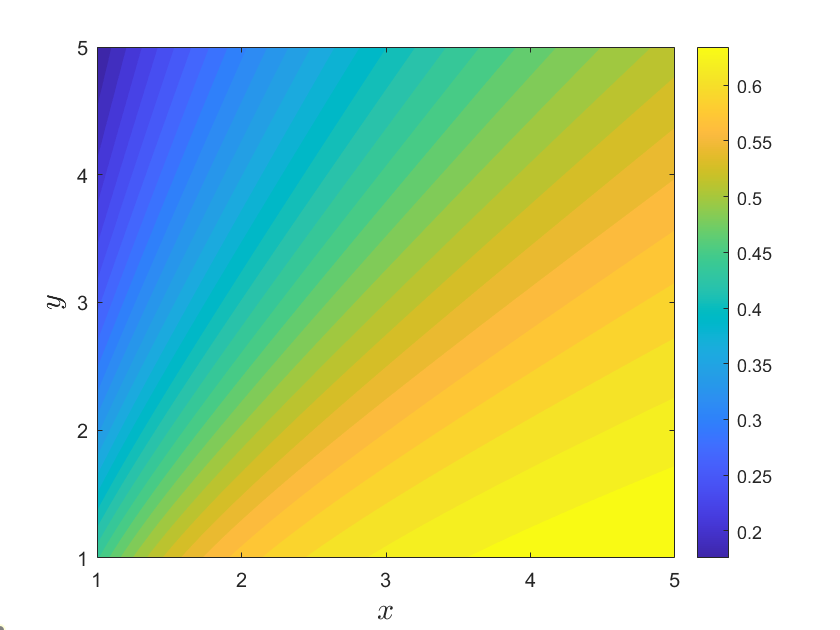}\label{fig:sub_4b}}
\end{minipage}
\caption{Solution profile of amplitude $u_5(x,y)$ for  $d_1\neq0$.}\label{fig:3aaa}
\begin{minipage}{.38\textwidth}
 \subfloat[]{\includegraphics[width=\textwidth]{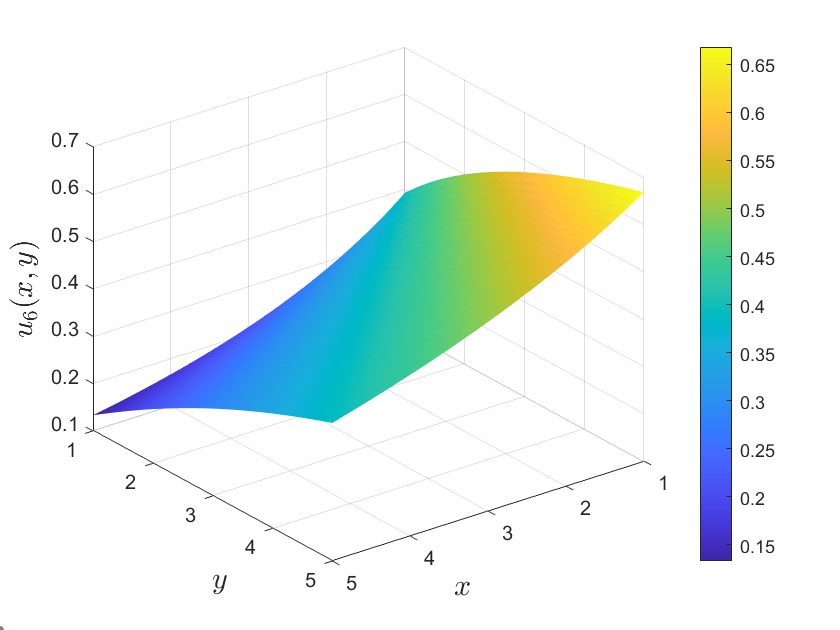}\label{fig:sub_4f}}
\end{minipage}
\begin{minipage}{.38\textwidth}
    \subfloat[]{\includegraphics[width=\textwidth]{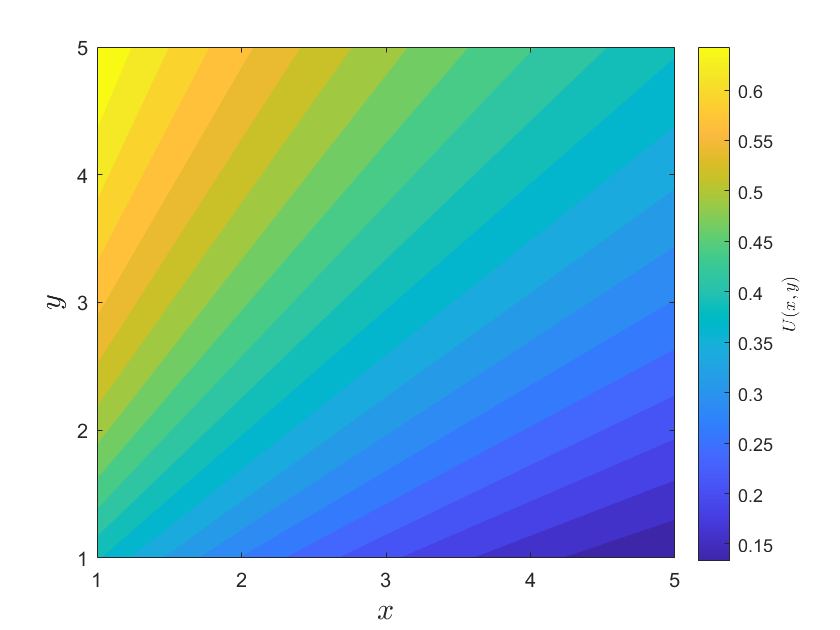}\label{fig:sub_4g}}
\end{minipage}
  \caption{Solution profile of amplitude $u_6(x,y)$ for $d_3\neq0$.}\label{fig:4.1}
\end{figure}
\section{Traveling Wave Solution}\label{7}
To find traveling wave solution of \eqref{11} , we assume the transformation as,
\begin{equation}
    \eta=ax+by+cz-dt, \hspace{2em} u(x,y,z,t)= K(\eta).
\end{equation}
Substituting the above values in the governing model \ref{11}, we get
\begin{equation}
    -d \frac{d K(\eta)}{d \eta}+a^2 k(\eta) \frac{d K(\eta)}{d \eta}+a^2 k(\eta) \frac{d^{2} K(\eta)}{d \eta^2}+b^2 k(\eta) \frac{d^{2} K(\eta)}{d \eta^2}+c^2 k(\eta) \frac{d^{2} K(\eta)}{d \eta^2}=0,
\end{equation}
On solving above ODE we get,
\begin{equation}
    K(\eta)=\frac{a^2 b+ \tanh({\frac{(C_2+\eta){\sqrt{2 C_1 a^8+2C_1 a^6b^2+2C_1a^6c^2}}}{2a^2(a^2+b^2+c^2)}) {\sqrt{2 C_1 a^8+2C_1 a^6b^2+2C_1a^6c^2}}}}{a^4},
\end{equation}
which provides the solution of \ref{11} as,
\begin{equation}\label{5.1}
   u_7(x,y,z,t)= \frac{a^2 b+ \tanh({\frac{(C_2+ax+by+cz-dt){\sqrt{2 C_1 a^8+2C_1 a^6b^2+2C_1a^6c^2}}}{2a^2(a^2+b^2+c^2)}) {\sqrt{2 C_1 a^8+2C_1 a^6b^2+2C_1a^6c^2}}}}{a^4},
\end{equation}
where $C_1$ and $C_2$ are integration constants. 
\begin{figure}[H]
\centering
\begin{minipage}{.38\textwidth}
    \subfloat[]{\includegraphics[width=\textwidth]{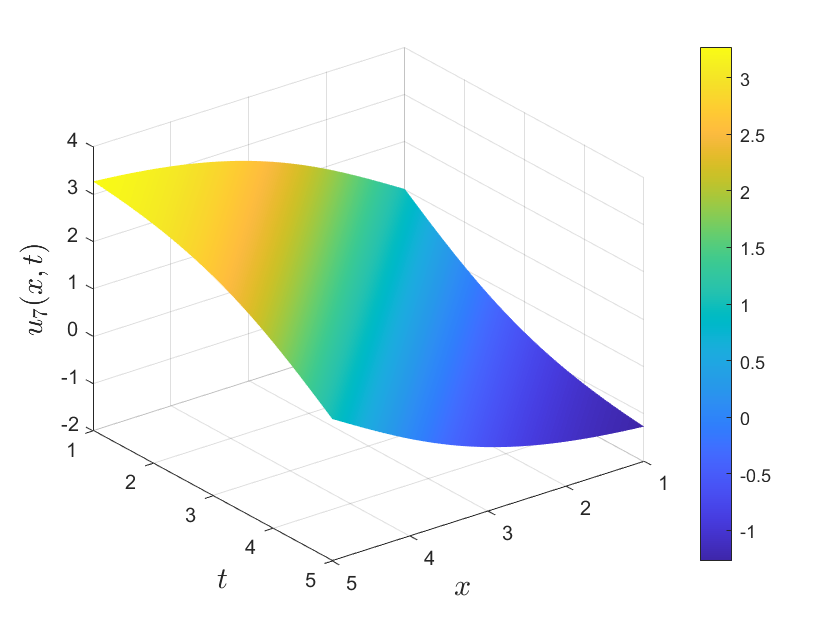}\label{fig:sub_5a}}
\end{minipage}
\begin{minipage}{.38\textwidth}
    \subfloat[]{\includegraphics[width=\textwidth]{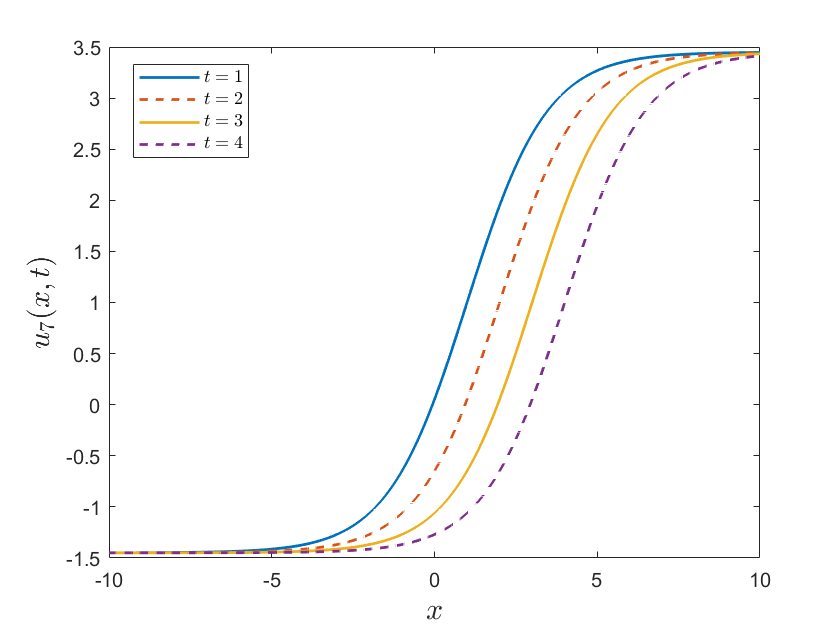}\label{fig:sub_5b}}
\end{minipage}\\
    \caption{3D and 2D plots for $u_7(x,t)$ corresponding to \eqref{5.1}.}\label{fig:5}
    \end{figure}
    \begin{itemize}
        \item Figures $7(a)$ and $7(b)$ show the 3D and 2D plots for $u$, respectively, that demonstrate how the wave retains its shape as it moves along the $x$ axis over $t$. It represents the kink-type solitary wave.
   \end{itemize} 
  
\section{Modulation Instability}\label{8}
To understand the nature of stability and the behavior of soliton structures, the concept of modulation instability is explored in this section. We analyze the concept of modulation instability of the traveling wave solution by using the dispersion relation and gain spectrum, which shows the circumstances under which minor perturbation grows, creating complex structures, indicating instability and the possibility of soliton formation. This study, in particular, makes it easy to understand how the different parameters affect the wave's evolution, which is necessary to understand the dynamic stability of solitons in practical applications.\\
To execute modulation stability \cite{qawaqneh2024stability} ; 
\begin{equation}\label{8.1}
    v=(u+\sqrt{A})e^{iAt},
\end{equation}
where $A$ is a sensitivity parameter.
Substituting \eqref{8.1} in \eqref{11} and after linearizing, we get,
\begin{equation}\label{8.2}
    iAu+iA^{3/2}+u_t+a\sqrt{A} e^{iAt} u_x+u_{xx}+u_{yy}+u_{zz}=0.
\end{equation}
Consider the solution of \eqref{8.2} as:
\begin{equation}\label{8.3}
 u(x,y,z,t)=a_1 e^{i(bx+cy+dz-wt)}+a_2 e^{-i(bx+cy+dz-wt)},  
\end{equation}
where $b, c, d$ are wave numbers and $w$ is the angular frequency.
 On Substituting \eqref{8.3} into \eqref{8.2} and then collecting the coefficient of $e^{i(bx+cy+dz-wt)}$ and $e^{-i(bx+cy+dz-wt)}$, we get the following dispersion relation by solving the determinant of the coefficient matrix
\begin{equation}
    w=\sqrt{A^2-(b^2+c^2+d^2)},
\end{equation}
where A is a parameter related to the amplitude or characteristic frequency, and $p=\sqrt{b^2+c^2+d^2}$ is a coefficient that influences the system. The gain spectrum, which affects the growth rate 
of perturbations, is defined as
\begin{equation}
    G(A)= 2Im(w).
\end{equation}
 It is important to differentiate between the stable and unstable regions for stability analysis. 
 \begin{itemize}
     \item When $A^2>p^2=b^2+c^2+d^2$, a real value of $w$ is obtained by the dispersion relation, which indicates that the system is stable and there is no modulation instability. Here, the gain spectrum is zero because there are no growing modes.
     \item Conversely, when  $A^2<p^2=b^2+c^2+d^2$, an imaginary value of $w$ is obtained by the dispersion relation, leading to modulation instability.
 \end{itemize} Here, the gain spectrum is given by
\begin{equation}
    G(A)=2Im(w)=2\sqrt{p^2-A^2}.
\end{equation}
\begin{figure}[H]
\centering
\begin{minipage}{.38\textwidth}
    \subfloat[]{\includegraphics[width=\textwidth]{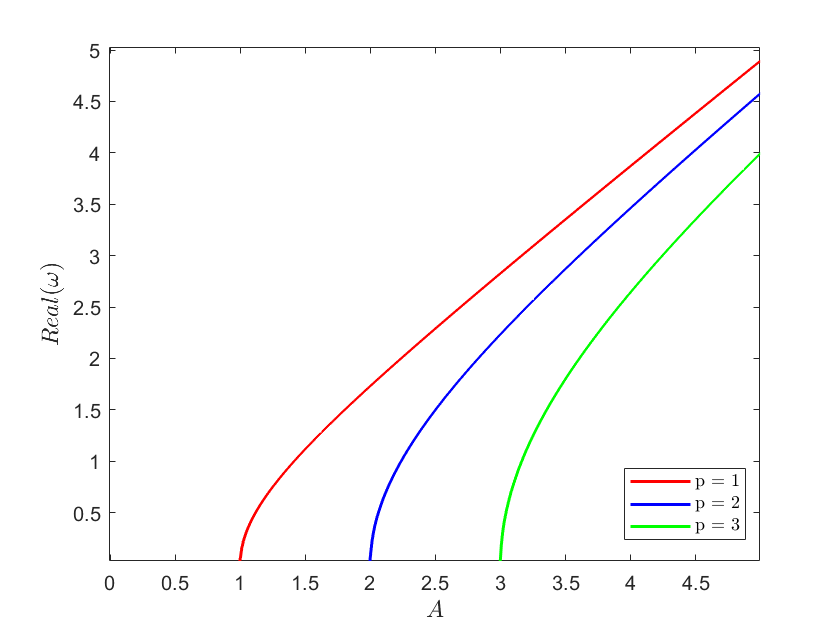}\label{fig:sub_6a}}
\end{minipage}
\begin{minipage}{.38\textwidth}
    \subfloat[]{\includegraphics[width=\textwidth]{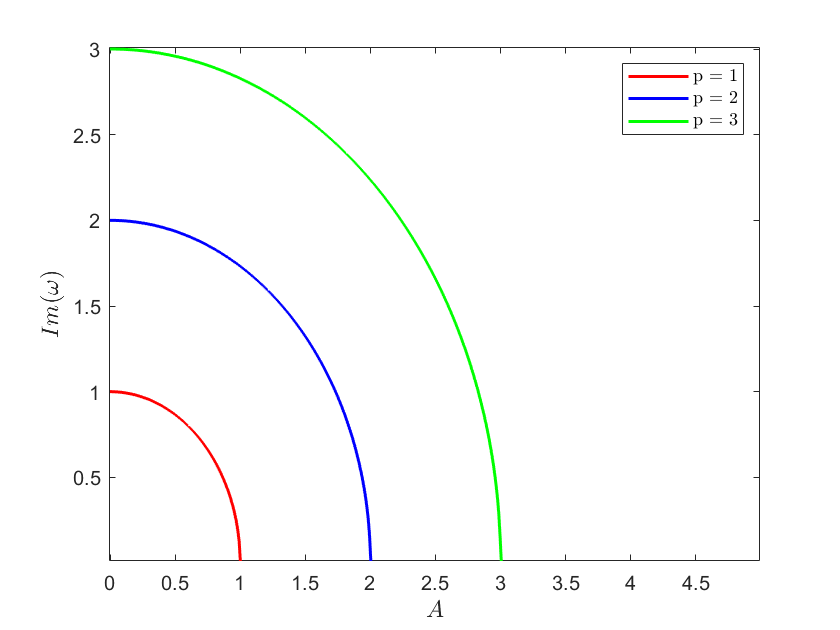}\label{fig:sub_6b}}
\end{minipage}\\
 \caption{Frequency of $w$ for distinct $p's$ (a) $A\geq p$ and (b) $A<p$.}\label{fig:6}
\end{figure}
\begin{figure}[H]
\centering
\begin{minipage}{.38\textwidth}
    \subfloat[]{\includegraphics[width=\textwidth]{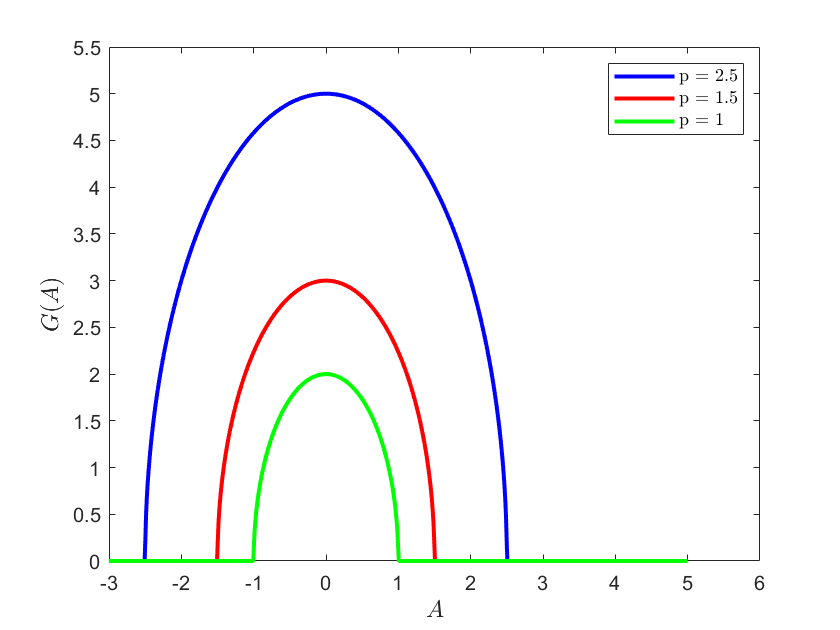}\label{fig:sub_7a}}
\end{minipage}\\
    \caption{Gain spectrum for distinct $p's$.}\label{fig:7}
    \end{figure}
        \begin{itemize}
        \item Figure $8(a)$ shows the curve for values of A that exceed $p$, where $w$ is a real number, and the graph illustrates positive values of $w$. Figure $8(b)$ shows the curve for values of A that are less than $p$, where $w$ is an imaginary number, and the graph illustrates negative values of $w$. These visualizations show the transition of the system between the stable and unstable regions. 

        Two different scenarios are shown in Figure $9$. Plotting the gain spectrum reveals non-zero values, which indicate modulation instability. The gain spectrum rises with $p$, indicating that disturbances are becoming more frequent. Stability is indicated by the zero-gain spectrum. In this instance, the plot displays a horizontal line at zero across the range of $p$, indicating that the system does not display modulation instability.
    \end{itemize}
    \section{Numerical Method and Validation} \label{9}
    To verify the accuracy and consistency of the invariant solutions of the (3+1)-dimensional ZK equation, obtained from the Lie symmetry reductions, we solve the resulting nonlinear ODEs using two widely used numerical methods: the classical fourth-order Runge-Kutta (RK4) method and MATLAB's adaptive ode45 solver. 

    We apply the RK method to each similarity-reduced ODE with appropriate initial conditions. We also solve these ODEs using MATLAB's ode45, based on an adaptive Runge-Kutta-Fehlberg (RKF) scheme of order 4(5) to cross-validate the results. To maintain control of error, this solver adjusts the step size accordingly, ensuring the accuracy and stability of ODEs.

   Both methods were used on the same equations with the same parameter value and initial conditions in our simulation. As illustrated in the comparison plots ( \ref{fig:11}, \ref{fig:12}, \ref{fig:13}, \ref{fig:14}, \ref{fig:15} ), the solution profile of the dependent variable and its derivative $H(x)$ and $H'(x)$ obtained from RK4 and ode45 show great agreement. The correctness of our RK implementation is confirmed by the overlapping of these graphs. 
   \section{Comparative Analysis of Analytical and Numerical Results} \label{10}
     This section is devoted to the comparison between the analytical solutions of the reduced ODEs using Lie symmetry reduction and the numerical counterparts obtained using the RK4 method. For fair comparison, the values of the parameters used in both analytical and numerical solutions are identical. For all figures, the solid and dashed lines represent the analytical and numerical solutions, respectively. The parameters used for the respective ODEs are as follows:
     \begin{itemize}
         \item Figure \ref{fig:11}: $a=1; b_2=1; b_5=0; b_6=1.$
         \item Figure \ref{fig:12}: $a= 1; b_1 = 1; b_2 = 2; b_3 = 1; A_1 = 1; A_2 = 1; d_2 = 1; d_4 = 1; d_7 = 1.$
         \item Figure \ref{fig:13}: $a=1; \lambda=1.$
         \item Figure \ref{fig:14}: $a=1; A_1=1; \lambda=1.$
         \item Figure \ref{fig:15}: $a=1; A_1=1; \lambda=1.$
     \end{itemize}
    In all cases, the comparison plots ( \ref{fig:11}, \ref{fig:12}, \ref{fig:13}, \ref{fig:14}, \ref{fig:15} ) illustrate that the analytical and numerical plots are almost indistinguishable; that confirms the stability and correctness of the solutions. 
   \begin{figure}[H]
\centering
\begin{minipage}{.35\textwidth}
    \subfloat[]{\includegraphics[width=\textwidth]{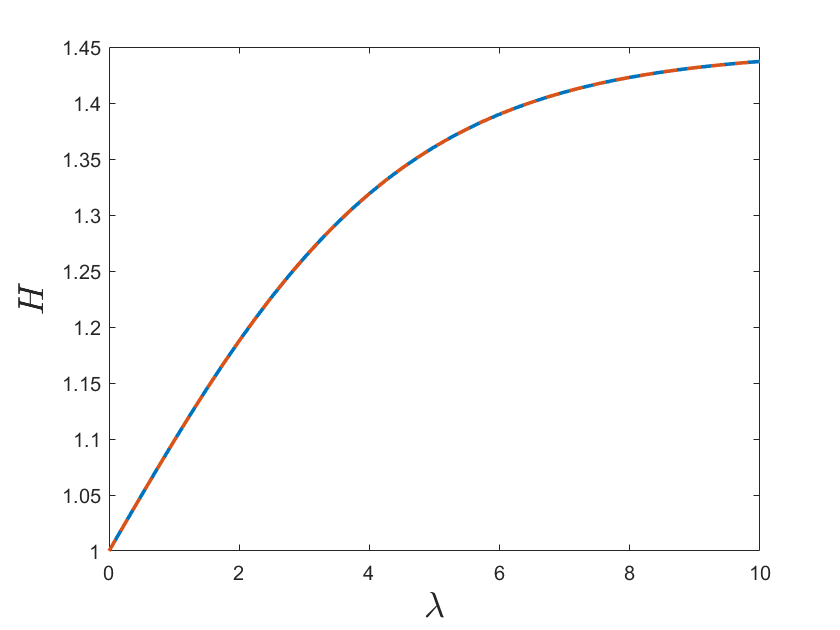}\label{fig:sub_11a}}
\end{minipage}
\begin{minipage}{.35\textwidth}
    \subfloat[]{\includegraphics[width=\textwidth]{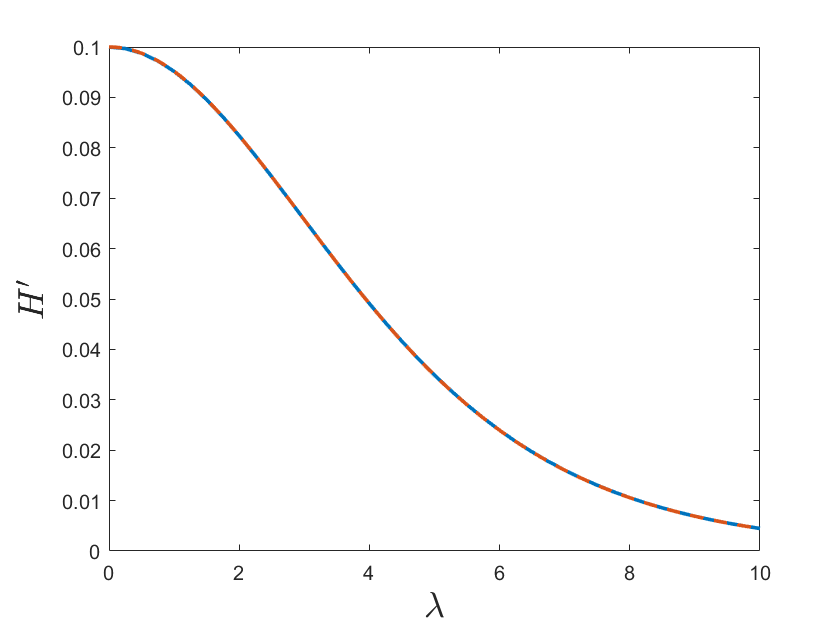}\label{fig:sub_11b}}
\end{minipage}\\
    \caption{ODE plot of $H(\lambda)$ and $H'(\lambda)$ for symmetry ${X}=b_{2} \mathfrak{D}_{2}+\mathfrak{D}_{3}+b_{5} \mathfrak{D}_{5}+b_6 \mathfrak{D}_{6}$    \ref{S1}.}\label{fig:11}
    \begin{minipage}{.35\textwidth}
    \subfloat[]{\includegraphics[width=\textwidth]{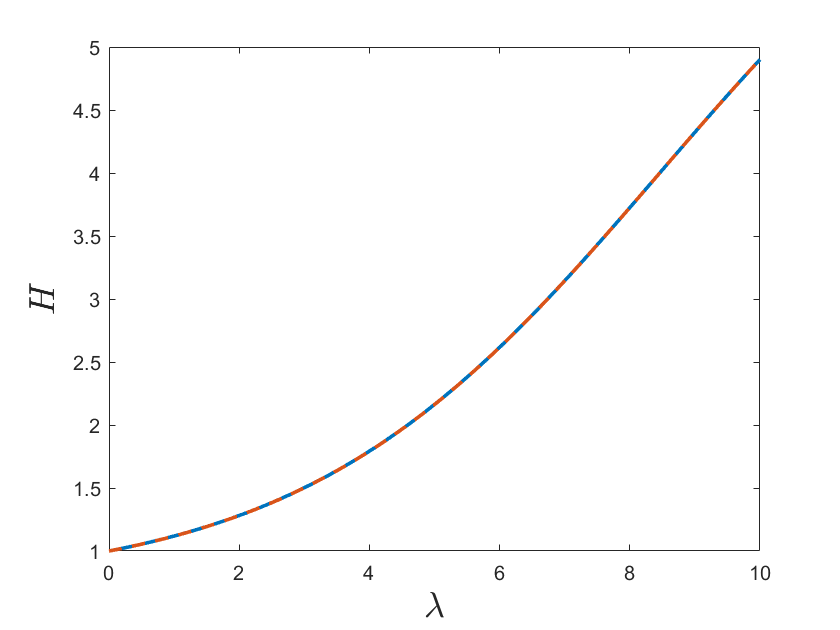}\label{fig:sub_12a}}
\end{minipage}
   \begin{minipage}{.35\textwidth}
    \subfloat[]{\includegraphics[width=\textwidth]{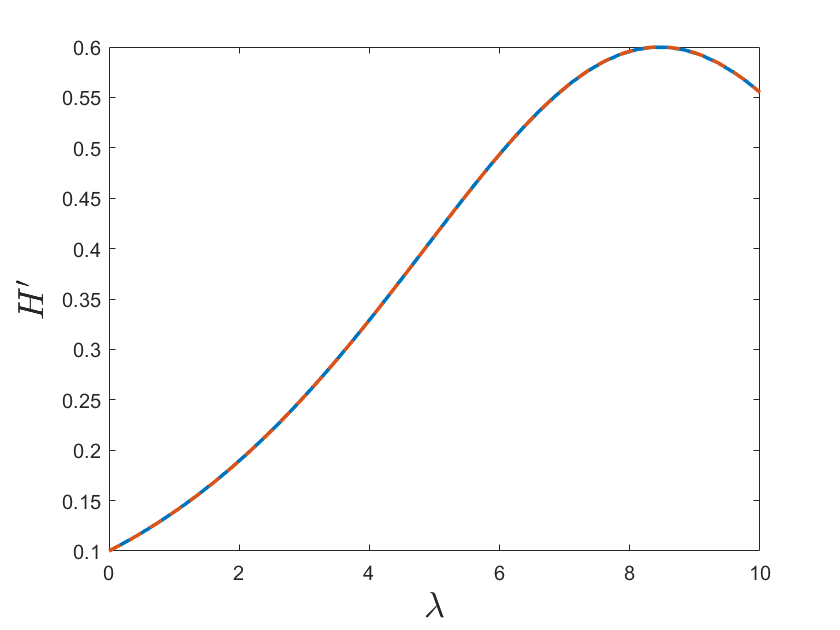}\label{fig:sub_12b}}
\end{minipage}\\
    \caption{ODE plot of $H(\lambda)$ and $H'(\lambda)$ for symmetry ${X}=d_{2} \mathfrak{D}_{2}+d_{4} \mathfrak{D}_{4}+d_{5} \mathfrak{D}_{5}+d_{6} \mathfrak{D}_{6}+d_{7} \mathfrak{D}_{7}$ \ref{S3}.}\label{fig:12}
     \end{figure}
    \begin{figure}[H]
\centering
    \begin{minipage}{.35\textwidth}
    \subfloat[]{\includegraphics[width=\textwidth]{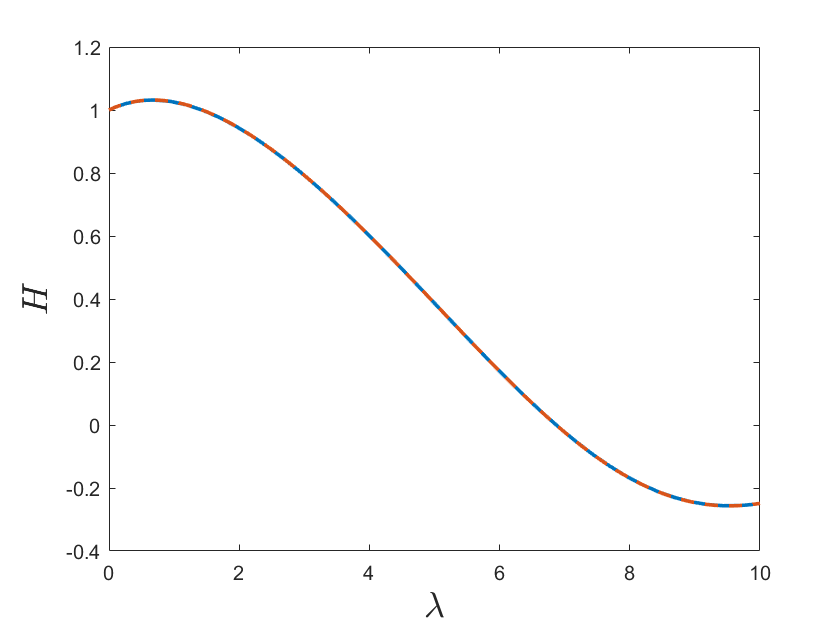}\label{fig:sub_13a}}
\end{minipage}
 \begin{minipage}{.35\textwidth}
    \subfloat[]{\includegraphics[width=\textwidth]{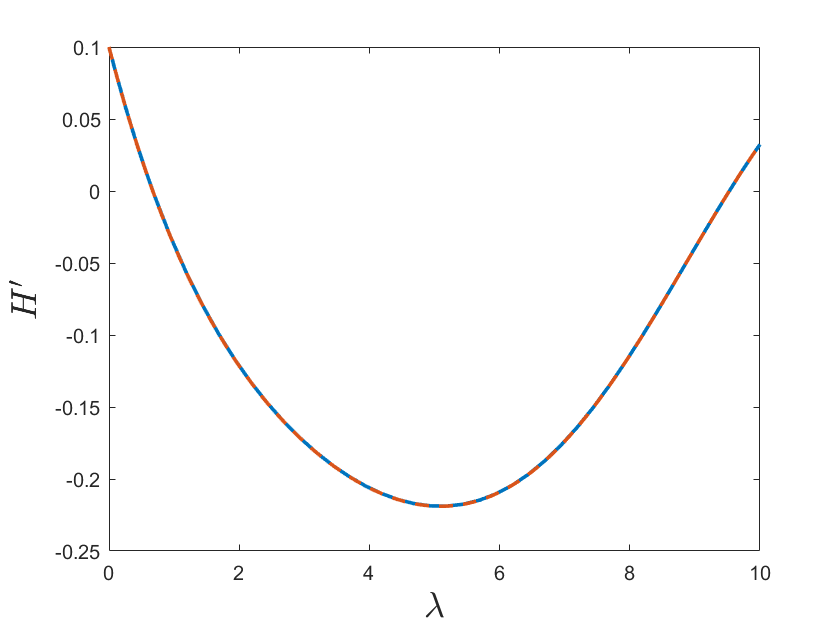}\label{fig:sub_13b}}
\end{minipage}\\
    \caption{ODE plot of $H(\lambda)$ and $H'(\lambda)$ for translation symmetry $\mathcal{D}_{2}$ for sub-case $d_1\neq 0$ \ref{S5}.}\label{fig:13}
    \end{figure}
    \begin{figure}[H]
\centering
\begin{minipage}{.35\textwidth}
 \subfloat[]{\includegraphics[width=\textwidth]{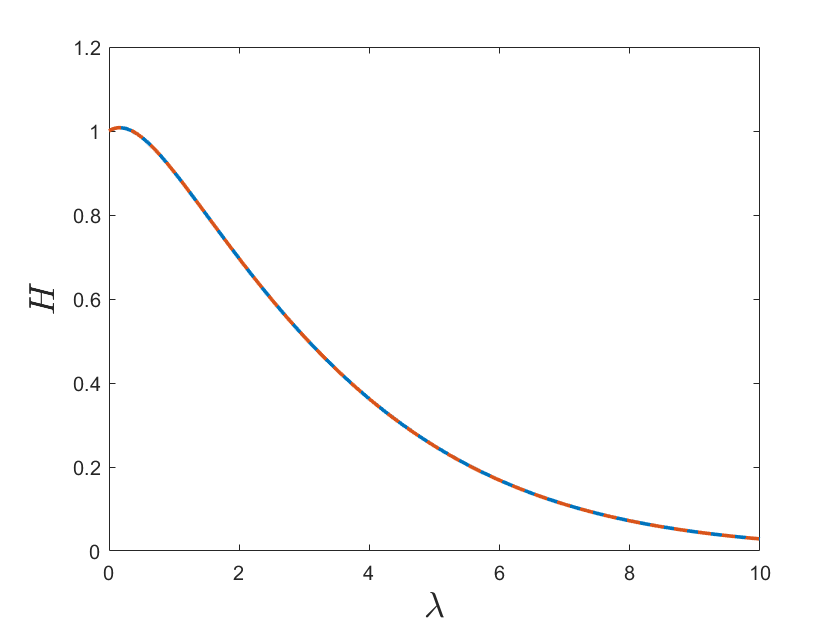}\label{fig:sub_14a}}
\end{minipage}
\begin{minipage}{.35\textwidth}
    \subfloat[]{\includegraphics[width=\textwidth]{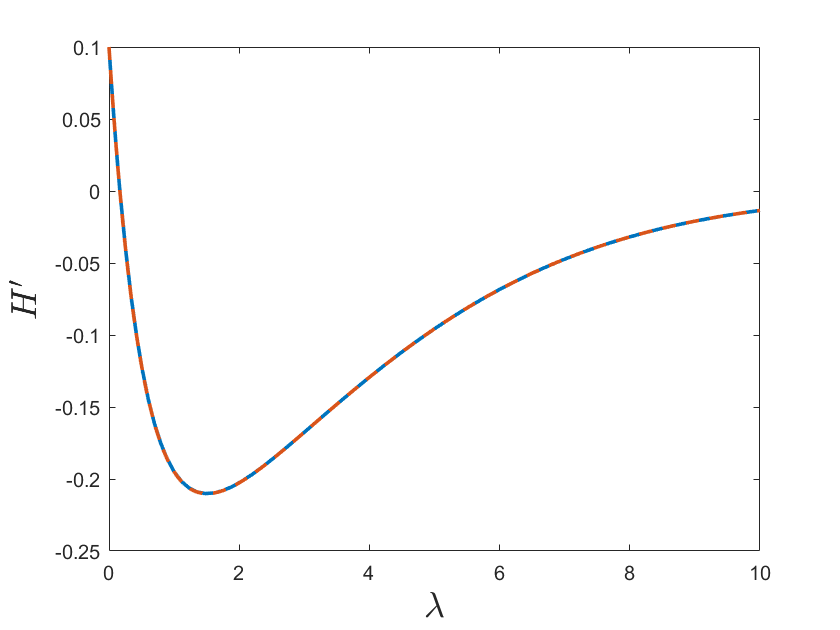}\label{fig:sub_14b}}
\end{minipage}
  \caption{ODE plot of $H(\lambda)$ and $H'(\lambda)$ for translation symmetry $\mathcal{D}_{2}$ for sub-case $d_2\neq 0$ \ref{S5}. }\label{fig:14}
\end{figure}

 \begin{figure}[H]
\centering
\begin{minipage}{.35\textwidth}
 \subfloat[]{\includegraphics[width=\textwidth]{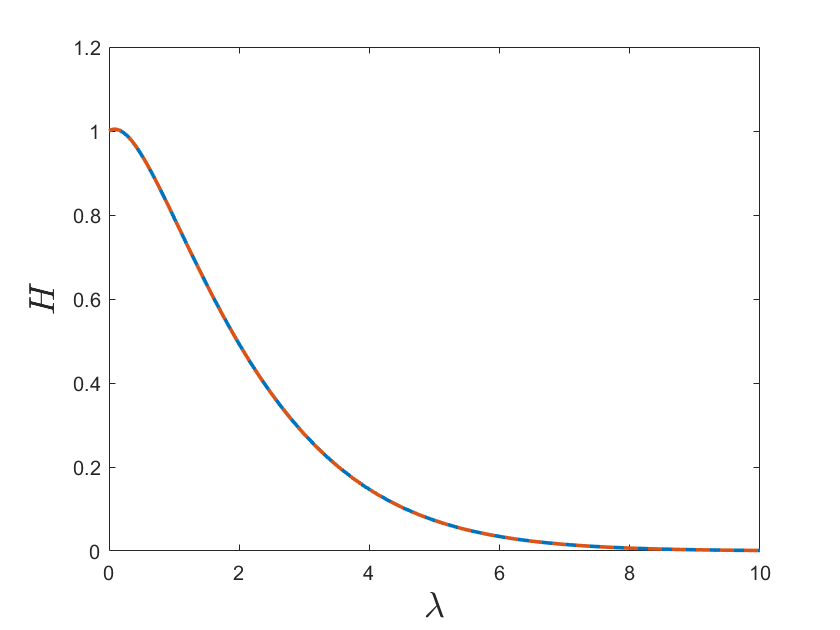}\label{fig:sub_15a}}
\end{minipage}
\begin{minipage}{.35\textwidth}
    \subfloat[]{\includegraphics[width=\textwidth]{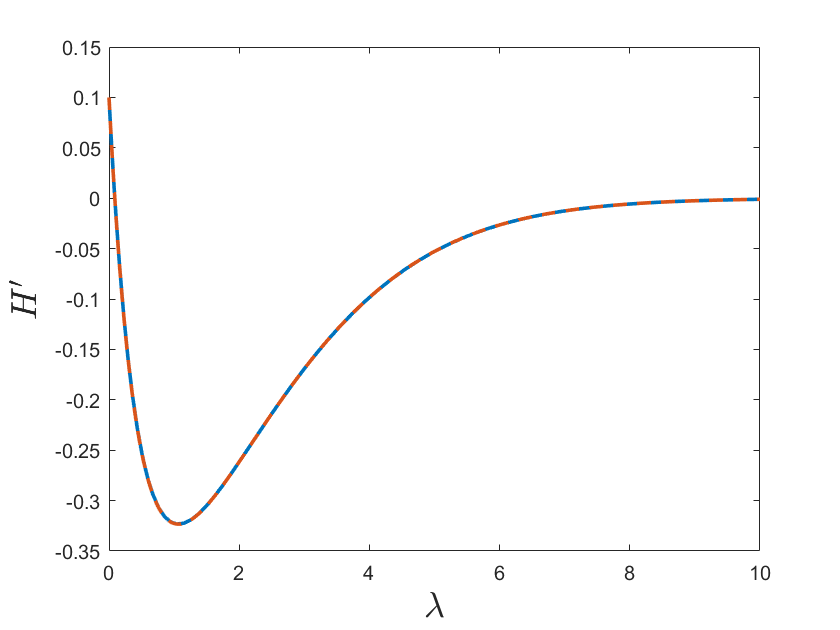}\label{fig:sub_15b}}
\end{minipage}
  \caption{ODE plot of $H(\lambda)$ and $H'(\lambda)$ for translation symmetry $\mathcal{D}_{2}$ for sub-case $d_3\neq 0$ \ref{S5}. }\label{fig:15}
\end{figure}
   
\section{Nonlinear Self-Adjointness for ZK Equation} \label{11a}
\quad \quad Ibragimov \cite{ibragimov2011nonlinear} introduced the concept of nonlinear self-adjointness for a given system of partial differential equations.
\begin{thm}
The given model Eq. \eqref{11} is nonlinearly self-adjoint.
\end{thm}
\begin{proof}
The adjoint equation corresponding to Eq.\eqref{11} can be written according to the procedure mentioned in \cite{ibragimov2011nonlinear}
\begin{equation}
S=\frac{\delta L}{\delta u}= -\frac{\partial \Psi}{\partial z} + \frac{\partial^2 \Psi}{\partial x^2}+ \frac{\partial^2 \Psi}{\partial y^2} + \frac{\partial^2 \Psi}{\partial z^2}\text{,}
\end{equation}
where $L$ is the formal Lagrangian defined as 
\begin{equation}
L=\Psi(x,y,z,t) \Delta, 
\end{equation}

where\hspace{0.1cm} $\Delta= u_{t}+auu_{x}+u_{x x}+u_{y y}+u_{z z}$ and $\Psi(x,y,z,t)$ is the dependent variable. 

The Euler-Lagrange operator $\frac{\delta }{\delta u}$ with respect to u is as follows:
\begin{equation}
\frac{\delta}{\delta u} = \frac{\partial}{\partial u} - D_t \frac{\partial}{\partial u_t} - D_x \frac{\partial}{\partial u_x} - D_y \frac{\partial}{\partial u_y}- D_z \frac{\partial}{\partial u_z} + D_x^2 \frac{\partial}{\partial u_{xx}} + D_x D_y \frac{\partial}{\partial u_{xy}}+ D_x D_z \frac{\partial}{\partial u_{xz}} + D_y^2 \frac{\partial}{\partial u_{yy}}+ D_z^2 \frac{\partial}{\partial u_{zz}}  + \cdots ,
\end{equation}
where $D_{x}$, $D_{y}$, $D_{z}$ and $D_{t}$ are total differential operators with respect to $x,y, z \hspace{0.1cm} \text{and} \hspace{0.1cm} t$.

The given Eq. \eqref{11} is nonlinearly self-adjoint \cite{ibragimov2011nonlinear} if it satisfies the following condition: 
\begin{equation}\label{e61}
    S \mid_{(\Psi=G(x,y,z,t,u))} = \delta_{1} \Delta\text{,}
\end{equation}
where $G(x,y,z,t,u)$ is nonzero function, $\delta_{1}$ is unknown to be determined by equating coefficient of $u_{t}$ to zero, and we obtain
\begin{equation}\label{lm}
    \delta_{1}=-\Psi_{u}\text{.}
\end{equation}
Substituting the value of $\delta_{1}$ into Eq.\eqref{lm}, one obtain
\begin{equation}\label{e63}
\Psi=(c_1 y z+ c_2 y+ c_3 z+ c_4)\text{.}
\end{equation}
Hence, the given Eq. \eqref{11} is nonlinearly self-adjoint under \eqref{e63}.
\end{proof}
\section{Conservation Laws} \label{12}
\quad \quad In mathematics, conservation laws are those mathematical expressions that illustrate the invariance of certain physical attributes, namely mass, momentum, and energy, concerning time. Conservation laws help us understand fluid flow and heat transfer, and analyze symmetry and efficiency. Conservation laws lead to the development of deriving PDEs, demonstrating the equilibrium of physical quantities. Ibragimov proposed the adjoint equation concept to obtain conservation laws.
\subsection{Construction of Conservation Laws Using Nonlinear Self-Adjointness}
\quad \quad Ibragimov \cite{ibragimov2011nonlinear} proposed the concept for the construction of conserved vectors. His work contributes to the advancement of understanding the relation between symmetries and conserved quantities. Conservation laws can be constructed by many methods, like the direct method, multiplier method, variational method, and many more, but Ibragimov's method has been considered a more effective method to construct conservation laws, as it applies to both variational and nonvariational systems.   \\
Let $\eta=(\eta^{x}, \eta^{y}, \eta^{z}, \eta^{t})$ be the conserved vector, satisfying the conservation form
\begin{equation}
    D_{x} \eta^{x}+D_{y} \eta^{y}+D_{z} \eta^{z}+D_{t} \eta^{t}=0\text{,}
    \end{equation}
and generated by the infinite symmetry 
\begin{equation}
    \mathbb{Z}=\xi_{1}\frac{\partial}{\partial x}+\xi_{2}\frac{\partial}{\partial y}+\xi_{3}\frac{\partial}{\partial z}+\xi_{4}\frac{\partial}{\partial t}+\xi_{5} \frac{\partial}{\partial u}\text{,}
\end{equation}
where the conserved vectors $\eta^{x},\eta^{y},\eta^{z}$ and $\eta^{t}$ can be obtained from the ideas developed in \cite{ibragimov2011nonlinear}
\begin{equation}\label{e766}
\begin{split}
& \eta^{x}=\xi_{1} L+\omega \frac{\partial L}{\partial u_{x}}-\omega D_{x} (\frac{\partial L}{\partial u_{xx}})+\frac{\partial L}{\partial u_{xx}}D_{x}(\omega)\text{,} \\
&\eta^{y}= \xi_{2} L+\omega \frac{\partial L}{\partial u_{y}}-\omega D_{y} (\frac{\partial L}{\partial u_{yy}})+\frac{\partial L}{\partial u_{yy}}D_{y}(\omega)\text{,}\\
& \eta^{z}= \xi_{3} L+\omega \frac{\partial L}{\partial u_{z}}-\omega D_{z} (\frac{\partial L}{\partial u_{zz}})+\frac{\partial L}{\partial u_{zz}}D_{z}(\omega)\text{,} \\  &\eta^{t}=\xi_{4} L+\omega \frac{\partial L}{\partial u_{t}}\text{,}
\end{split}
\end{equation}
with $\omega=\xi_{5}-\xi_{1}u_{x}-\xi_{2}u_{y}-\xi_{3}u_{z}-\xi_{4}u_{t}$.

Furthermore, we evaluate the conserved vectors for the infinitesimal generators, as in the following cases:
\renewcommand{\arraystretch}{1.2}
\setlength{\arrayrulewidth}{0.4mm}
\begin{table}[H]
    \centering
    \begin{tabular}{c|c}
        \hline
        \textbf{Lie Algebra} & \textbf{Conserved Vectors} \\
        \hline
        & $\eta^x_1 =\Psi \frac{x}{2}(u_t + u_{xx} + u_{yy} + u_{zz}) + a\bigg(-\frac{u}{2} - \frac{x}{2} u_x - \frac{y}{2} u_y - \frac{z}{2} u_z- t u_t \bigg)\Psi u$ \\ 
        & $ + \big( -u_x - x u_{xx} - y u_{xy} - z u_{xz} - t u_{tx} \big)\Psi$, \\
        $\mathfrak{D}_{1}$ & $\eta^y_1 = \frac{y}{2}\Psi(u_t + au_x + u_{yy} + u_{zz}) + \bigg(\frac{u}{2} + \frac{x}{2} u_x + \frac{y}{2} u_y + \frac{z}{2} u_z + t u_t \bigg)(c_1 z + c_2)$ \\
        & $ - (u_y + \frac{x}{2} u_{xy} + \frac{z}{2} u_{yz} + t u_y)\Psi$, \\
        & $\eta^z_1 = \frac{z}{2}\Psi(u_t + au_x + u_{yy} + u_{zz}) + \bigg(\frac{u}{2} + \frac{x}{2} u_x + \frac{y}{2} u_y + \frac{z}{2} u_z + t u_t \bigg)(c_1 y + c_3)$ \\ 
        & $ - (u_z + \frac{x}{2} u_{xz} + \frac{y}{2} u_{yz})\Psi$, \\
        & $\eta^t_1 = - \bigg(\frac{u}{2} + \frac{x}{2} u_x + \frac{y}{2} u_y + \frac{z}{2} u_z + t u_t \bigg)\Psi + t \Psi (u_t + au_x + u_{yy} + u_{zz})$. \\ 
        \hline
        & $\eta^x_2 =-a\Psi u u_{t} -\Psi u_{t x}$, \\
        $\mathfrak{D}_{2}$ & $\eta^y_2 = -\Psi u_{t y}+ (c_{1} z + c_{2}) u_{t}$, \\ 
        & $\eta^z_2 = -\Psi u_{t z}+ (c_{1} y + c_{3}) u_{t}$, \\
        & $\eta^t_2 = -\Psi u_{t}+\Psi (u_{t}+ a u{} u_{x}+ u_{x x}+ u_{y y}+ u_{z z})$. \\ 
        \hline
        & $\eta^x_3 =au\Psi(-z u_{y}+ y u_{z})+\Psi(-z u_{x y}+ y u_{x z})$, \\
        $\mathfrak{D}_{3}$ & $\eta^y_3= \Psi(-z u_{y y}+ u_{z}+ y u_{y z})-(c_{1} z + c_{2})(-z u_{y}+ y u_{z})$ \\
        &$z\Psi (u_{t} + a u{} u_{x}+u_{x x} + u_{y y} + u_{z z})$, \\
        & $\eta^z_3 = \Psi(-u_{y}-z u_{y z}+ y u_{z z})-(c_{1} y + c_{3})(-z u_{y}+ y u_{z})$\\
        &$-y\Psi (u_{t} + a u{} u_{x}+u_{x x} + u_{y y} + u_{z z})$,\\
        & $\eta^t_3 = \Psi(-z u_{y}+y u_{z})$. \\ 
        \hline
         & $\eta^x_4 =-a\Psi u u_{y}- u_{x y}\Psi$,\\
        $\mathfrak{D}_{4}$ & $\eta^y_4 = -\Psi u_{y y}+(c_{1} z + c_{2}) u_{y}+\Psi(u_{t} + a u{} u_{x} + u_{x x} + u_{y y} + u_{z z})$, \\ 
        & $\eta^z_4 =-\Psi u_{y z}+(c_{1} y + c_{3}) u_{y}$, \\
        & $\eta^t_4 =-\Psi u_{y}$. \\ 
        \hline
        & $\eta^x_5 =a\Psi u{} (1/a- t u_{x})-(c_1 y z + c_2 y + c_3 z + c_4)t u_{x x}+ \Psi t (u_{t} + a u{} u_{x} +u_{x x} + u_{y y} + u_{z z})$,\\
        $\mathfrak{D}_{5}$ & $\eta^y_5 = -\Psi t u_{x y} -(c_{1} z+ c_{2})(1/a- t u_{x})$, \\ 
        & $\eta^z_5 = -\Psi t u_{x z} -(c_{1} y+ c_{3})(1/a- t u_{x})$, \\
        & $\eta^t_5 =\Psi(1/a- t u_{x})$. \\ 
         \hline
          & $\eta^x_6 =a\Psi u u_{x}+ u_{x x}\Psi+ \Psi(u_{t} + a u{} u_{x} + u_{x x} + u_{y y} + u_{z z})$,\\
        $\mathfrak{D}_{6}$ & $\eta^y_6 = \Psi u_{x y}-(c_{1} z + c_{2}) u_{x}$, \\ 
        & $\eta^z_6 = \Psi u_{x z}-(c_{1} y + c_{3}) u_{x}$, \\
        & $\eta^t_6 =\Psi u_{x}$. \\
        \hline
          & $\eta^x_7 =a\Psi u u_{z}+ u_{x z}\Psi$,\\
        $\mathfrak{D}_{7}$ & $\eta^y_7= \Psi u_{y z}-(c_{1} z + c_{2}) u_{z}$, \\ 
        & $\eta^z_7 = \Psi u_{z z}-(c_{1} y + c_{3}) u_{z}+ \Psi(u_{t} + a u{} u_{x} + u_{x x} + u_{y y} + u_{z z})$, \\
        & $\eta^t_7 = \Psi u_{z}$. \\ 
        \hline
         \end{tabular}
    \caption{Lie Algebra and Conserved Vectors}
    \label{tab:conserved_vectors}
    \end{table}
     \section{Comparison with MSE Method} \label{13}
  Kamruzzaman Khan \cite{khan2014exact} applied the Modified Simple Equation (MSE) method to obtain exact traveling-wave solutions of the (3+1)-dimensional Zakharov–Kuznetsov (ZK) equation. The solutions obtained via MSE correspond to particular traveling-wave profiles that can be recovered as special cases of the more general invariant solutions derived through the Lie symmetry approach. In our analysis, the Lie symmetry method yields a larger set of symmetry generators and corresponding invariants, enabling the construction of a broader family of exact solutions. The MSE solutions arise when specific parameter constraints are imposed on the general Lie-invariant solutions. This relationship not only shows that MSE is contained within the Lie framework but also allows a direct numerical comparison on a common reduced form of the equation. Such comparisons confirm that the MSE profiles match the corresponding Lie-invariant solutions within numerical precision, while the Lie approach additionally uncovers more symmetries and solution structures, offering a deeper understanding of the equation's dynamics.
\section{Conclusions} \label{14}
\begin{sloppypar}
\quad \quad  In this research paper, the behavior of wave amplitude, propagation of soliton, and anisotropic behavior in a strong magnetized field using the ZK equation has been effectively analyzed. Infinitesimal generators and vector fields have been successfully derived using a one-parameter Lie group of transformations. Using these infinitesimal generators, we successfully constructed a commutator table followed by the adjoint representation table to find invariants of the Zakharov–Kuznetsov (ZK) equation. Many cases have been created using these invariants to analyze how the wave behaves under a strong magnetic field. Different cases yield different solutions, which helps in understanding the propagation of solitons in different media. Additionally, the graphical representations, including 3D visualizations, contour sketches, and 2D visualizations, have been provided to better understand the propagation of the wave in a magnetized field. We constructed different graphs by taking different values of variables to check the anisotropic nature of space plasmas. A traveling wave solution has been derived to identify a soliton-like structure. Furthermore, we explore modulation instability using the dispersion relation, which helps in finding the stability and instability zones. To highlight the similarities and the differences, the solutions derived from Lie symmetry are compared with the MSE method. Additionally, the numerical method is implemented to validate the accuracy of the analytical solutions. Moreover, the non-linear self-adjointness property has been derived, which helps in the construction of conservation laws given by Ibragimov. The obtained graphs and their solutions help in better understanding the anisotropic environments and wave behavior with applications in both fusion systems and space plasma. The results have applications in many fields, like space plasmas, ion-acoustic plasmas, dusty plasmas, etc. 
\end{sloppypar}
\section{Future Recommendation}
\quad \quad Future work can focus on extending the $(3+1)$-dimensional Zakharov–Kuznetsov equation to generalized forms with additional nonlinear terms or variable coefficients, along with numerical simulations to study wave stability and interactions. Investigating the effects of external forces, higher-order conservation laws, and coupled systems can provide deeper insights. Validating theoretical results through real-world applications in plasma physics, optics, and acoustics is another essential direction.
\section{Acknowledgements}
\quad \quad The first author, Anshika Singhal, and the corresponding author, Rajan Arora, acknowledge the financial support from the Council of Scientific and Industrial Research (CSIR), New Delhi, India, for the sponsored project with sanction order No. 25/0327/23/EMR-II dated 03/07/2023. The second author, Urvashi Joshi, acknowledges the financial support from the Ministry of Education (MoE).
\section{Conflict of Interest}
According to the author's report, there is no conflict of interest associated with this article.
\bibliographystyle{elsarticle-num}
\bibliography{ZK.bib}

\begin{thebibliography}{10}
\expandafter\ifx\csname url\endcsname\relax
  \def\url#1{\texttt{#1}}\fi
\expandafter\ifx\csname urlprefix\endcsname\relax\def\urlprefix{URL }\fi
\expandafter\ifx\csname href\endcsname\relax
  \def\href#1#2{#2} \def\path#1{#1}\fi

\bibitem{zakharov1974three}
V.~Zakharov, E.~Kuznetsov, On three dimensional solitons, Zhurnal Eksp. Teoret. Fiz 66 (1974) 594--597.

\bibitem{wazwaz2010partial}
A.-M. Wazwaz, Partial differential equations and solitary waves theory, Springer Science \& Business Media, 2010.

\bibitem{islam2015generalized}
M.~S. Islam, K.~Khan, A.~H. Arnous, Generalized kudryashov method for solving some (3+ 1)-dimensional nonlinear evolution equations, New Trends in Mathematical Sciences 3~(3) (2015) 46.

\bibitem{khan2014exact}
K.~Khan, M.~A. Akbar, Exact solutions of the (2+ 1)-dimensional cubic klein--gordon equation and the (3+ 1)-dimensional zakharov--kuznetsov equation using the modified simple equation method, Journal of the Association of Arab Universities for Basic and Applied Sciences 15 (2014) 74--81.

\bibitem{niwas2024exploring}
M.~Niwas, S.~Kumar, R.~Rajput, D.~Chadha, Exploring localized waves and different dynamics of solitons in (2+ 1)-dimensional hirota bilinear equation: a multivariate generalized exponential rational integral function approach, Nonlinear Dynamics (2024) 1--14.

\bibitem{mathanaranjan2022effective}
T.~Mathanaranjan, An effective technique for the conformable space-time fractional cubic-quartic nonlinear {S}chrodinger equation with different laws of nonlinearity, Computational Methods for Differential Equations 10~(3) (2022) 701--715.

\bibitem{zhou2022auto}
T.-Y. Zhou, B.~Tian, C.-R. Zhang, S.-H. Liu, Auto-b{\"a}cklund transformations, bilinear forms, multiple-soliton, quasi-soliton and hybrid solutions of a (3+ 1)-dimensional modified {K}orteweg-de {V}ries-{Z}akharov-{K}uznetsov equation in an electron-positron plasma, The European Physical Journal Plus 137~(8) (2022) 912.

\bibitem{fritzsche1999sophus}
B.~Fritzsche, Sophus {L}ie, Journal of Lie Theory 9 (1999) 1--38.

\bibitem{helgason1992sophus}
S.~Helgason, Sophus {L}ie, the mathematician, in: The Sophus Lie Memorial Conference, Scandinavian Univ. Press Oslo, August, 1992, pp. 3--21.

\bibitem{schwarz1988symmetries}
F.~Schwarz, Symmetries of differential equations: from {S}ophus {L}ie to computer algebra, {SIAM} Review 30~(3) (1988) 450--481.

\bibitem{lou2005non}
S.-Y. Lou, H.-C. Ma, Non-{L}ie symmetry groups of (2+ 1)-dimensional nonlinear systems obtained from a simple direct method, Journal of Physics A: Mathematical and General 38~(7) (2005) L129.

\bibitem{zhai2019lie}
X.-H. Zhai, Y.~Zhang, Lie symmetry analysis on time scales and its application on mechanical systems, Journal of Vibration and Control 25~(3) (2019) 581--592.

\bibitem{gazizov1998lie}
R.~K. Gazizov, N.~H. Ibragimov, Lie symmetry analysis of differential equations in finance, Nonlinear Dynamics 17 (1998) 387--407.

\bibitem{kosmann2010groups}
Y.~Kosmann-Schwarzbach, et~al., Groups and {S}ymmetries, Universitext, Springer, New York (2010).

\bibitem{cantwell2004introduction}
B.~J. Cantwell, T.~Moulden, Introduction to symmetry analysis, Applied Mechanics Reviews 57~(1) (2004) B4.

\bibitem{bluman1990simplifying}
G.~Bluman, Simplifying the form of {L}ie groups admitted by a given differential equation, Journal of {M}athematical Analysis and {A}pplications 145~(1) (1990) 52--62.

\bibitem{tracina2014nonlinear}
R.~Tracin{\`a}, On the nonlinear self-adjointness of the {Z}akharov--{K}uznetsov equation, Communications in Nonlinear Science and Numerical Simulation 19~(2) (2014) 377--382.

\bibitem{kosmann2011noether}
Y.~Kosmann-Schwarzbach, B.~E. Schwarzbach, Y.~Kosmann-Schwarzbach, The {N}oether theorems, Springer, 2011.

\bibitem{byers1998noether}
N.~Byers, E. {N}oether's discovery of the deep connection between symmetries and conservation laws, arXiv preprint physics/9807044 (1998).

\bibitem{naz2012conservation}
R.~Naz, Conservation laws for some systems of nonlinear partial differential equations via multiplier approach, Journal of Applied Mathematics 2012~(1) (2012) 871253.

\bibitem{yasar2017symmetries}
E.~Yasar, Y.~Yildirim, Symmetries and conservation laws of evolution equations via multiplier and nonlocal conservation methods, New Trends in Mathematical Sciences 5~(1) (2017) 128--136.

\bibitem{ibragimov2011nonlinear}
N.~H. Ibragimov, Nonlinear self-adjointness and conservation laws, Journal of Physics A: Mathematical and Theoretical 44~(43) (2011) 432002.

\bibitem{ibragimov2013nonlinear}
N.~H. Ibragimov, E.~D. Avdonina, Nonlinear self-adjointness, conservation laws, and the construction of solutions of partial differential equations using conservation laws, Russian Mathematical Surveys 68~(5) (2013) 889.

\bibitem{anco2017incompleteness}
S.~C. Anco, On the incompleteness of {I}bragimov’s conservation law theorem and its equivalence to a standard formula using symmetries and adjoint-symmetries, Symmetry 9~(3) (2017) 33.

\bibitem{sharma2023invariance}
A.~K. Sharma, S.~Yadav, R.~Arora, Invariance analysis, optimal system, and group invariant solutions of (3+ 1)-dimensional non-linear {MA-FAN} equation, Mathematical Methods in the Applied Sciences 46~(17) (2023) 17883--17909.

\bibitem{hu2015direct}
X.~Hu, Y.~Li, Y.~Chen, A direct algorithm of one-dimensional optimal system for the group invariant solutions, Journal of Mathematical Physics 56~(5) (2015).

\bibitem{qawaqneh2024stability}
H.~Qawaqneh, J.~Manafian, M.~Alharthi, Y.~Alrashedi, Stability analysis, modulation instability, and beta-time fractional exact soliton solutions to the van der waals equation., Mathematics (2227-7390) 12~(14) (2024).

\end{thebibliography}
\end{document}